\documentclass{amsart}
\usepackage{yufei}
\renewcommand{\phi}{\varphi}
\definecolor{WHITE}{RGB}{0, 0, 0}
\newcommand{\surp}{\mathsf{sp}}
\newcommand{\mc}{\mathsf{mc}}

\title{An Alon--Boppana--type bound for very dense graphs, with applications to max--cut}
\author{Shengtong Zhang}
\address{Department of Mathematics, Stanford University, USA.}
\email{stzh1555@stanford.edu}
\begin{document}
\maketitle
\begin{abstract}
For any $\epsilon > 0$, we show that if $G$ is a regular graph on $n \gg_\epsilon 1$ vertices that is $\epsilon$-far (differs by at least $\epsilon n^2$ edges) from any Tur\'{a}n graph, then its second eigenvalue $\lambda_2$ satisfies
$$\lambda_2 \geq n^{1/4 - \epsilon}.$$   
The exponent $1/4$ is optimal. Our result generalizes an analogous bound, independently obtained by Balla, R\"{a}ty---Sudakov--Tomon, and Ihringer, which only applies to graphs with density at most $\frac{1}{2}$. Up to a lower--order factor, this confirms a conjecture of R\"{a}ty, Sudakov and Tomon.

Our spectral approach has interesting applications to max--cut. First, we show that if a graph $G$, on $n \gg_\epsilon 1$ vertices and $m$ edges, is $\epsilon$-far from a disjoint union of cliques, then it has a max--cut of size at least
$$\frac{m}{2} + n^{1.01}.$$
Our result improves upon a classical result of Edwards by a non--trivial polynomial factor, making progress towards another conjecture of R\"{a}ty, Sudakov and Tomon.

As another application of our method, we show that if a graph $G$ is $H$-free and has $m$ edges, then $G$ has a max--cut of size at least
$$\frac{m}{2} + c_H m^{0.5001}$$
where $c_H > 0$ is some constant depending on $H$ only. This result makes progress towards a conjecture of Alon, Bollob\'{a}s, Krivelevich and Sudakov, and answers recent questions by Glock--Janzer--Sudakov and Balla--Janzer--Sudakov.
\end{abstract}
\section{Introduction}
Spectral parameters of graphs encode a remarkable amount of combinatorial information. Consider a $d$-regular graph $G$ on $n$ vertices with adjacency matrix eigenvalues
\[
d=\lambda_{1}\ge\lambda_{2}\ge\cdots\ge\lambda_{n}.
\]
The \emph{Alon–Boppana bound} \cite{Alon1991} is one of the cornerstones of spectral graph theory. It states that, if $G$ has diameter $D$, then the second eigenvalue of $G$ is lower-bounded by
\begin{equation}
    \label{eq:Alon-Boppana}
  \lambda_{2} \geq 2\sqrt{d-1}- \frac{2\sqrt{d-1} - 1}{\floor{D / 2}}.  
\end{equation}
For sparse graphs with $d = n^{o(1)}$, we have $\lambda_2 = (2 - o(1)) \sqrt{d - 1}$. This bound plays an important role in the theory of expanders, leading to the rich studies of Ramanujan graphs and optimally pseudorandom $K_t$-free graphs.

The original Alon--Boppana bound \eqref{eq:Alon-Boppana} is useful for sparse graphs. It does not tell us much about $\lambda_2$ when $D \leq 3$, and thus is no longer interesting for dense graphs with $d = \omega(n^{1/3})$. Recently, three groups of authors independently discovered an analog of the Alon--Boppana bound in the dense regime. Their results state that, when $d \leq \left(\frac{1}{2} - \epsilon\right) n$, the second-eigenvalue of $G$ is lower bounded by
$$\lambda_2 = \Omega_\epsilon(n^{1/3}).$$
Their works illustrate the connections between this bound and a number of interesting combinatorial objects. Balla \cite{Balla2024} put this bound in the context of equiangular lines, R\"{a}ty--Sudakov--Tomon \cite{RST2023} showed how this bound arises from their study of discrepancy, and Ihringer \cite{Ihringer2023} proved this bound through classical techniques in the literature on distance regular graphs \cite{BCN89}. Later, Balla--R\"{a}ty--Sudakov--Tomon \cite{BRST} gave a short and purely linear algebraic proof of this bound. 

The assumption $d \leq \left(\frac{1}{2} - \epsilon\right) n$ is necessary. The complete bipartite graphs $K_{n/2, n/2}$, and in general all balanced Tur\'{a}n graphs, satisfy $\lambda_2 = 0$. In addition, Weyl's inequality shows that if a small number of edges are added or deleted from a Tur\'{a}n graph, the resulting graph still has small second eigenvalue.

Are the Tur\'{a}n graphs the only obstruction to an Alon--Boppana bound? More precisely, we say two graphs $G$ and $H$ are \emph{$\epsilon$-close} if starting from $G$, we can get a graph isomorphic to $H$ by removing or adding at most $\epsilon |V|^2$ edges. We say they are \emph{$\epsilon$-far} otherwise. Motivated by their work on discrepancy theory, R\"{a}ty, Sudakov and Tomon made a bold conjecture: an analog of the Alon--Boppana bound holds if $G$ is far from any Tur\'{a}n graph.
\begin{conjecture}[\cite{RST2023}, Conjecture 8.2]
    \label{conj:second-eigenvalue}
    For any $\epsilon > 0$, there exists a $c > 0$ such that the following holds. If $G$ is a regular graph on $n$ vertices that is $\epsilon$-far from any Tur\'{a}n graph (including the empty one), then 
    $$\lambda_2 \geq cn^{1/4}.$$
\end{conjecture}
R\"{a}ty, Sudakov and Tomon observed that the exponent $1/4$ is optimal by a construction of de Caen \cite{deCaen}. Passing to the complement, the conjecture can be reformulated as follows.
\begin{conjecture}
    \label{conj:least-eigenvalue}
    For any $\epsilon > 0$, there exists a $c > 0$ such that the following holds. If $G$ is a regular graph on $n$ vertices that is $\epsilon$-far from a disjoint union of cliques, then 
    $$\lambda_n \leq -cn^{1/4}.$$
\end{conjecture}
We note that there is a large body of work on graphs with bounded least eigenvalue $\lambda_n = O(1)$; see for example the survey \cite{KCY2021}. In particular, graphs with least eigenvalue at least $-2.019$ admit explicit classifications \cite{AcharyaJiang24, Eigenvalueneg2}, and Kim--Koolen--Yang \cite{KKY2016} developed a structure theorem for graphs with bounded $\lambda_n$. This structure theorem can show bounds such as $\lambda_n \to -\infty$ in \cref{conj:least-eigenvalue}. However, they remarked that their Ramsey--theoretic approach might not be able to give polynomial--type bounds. For strongly regular graphs (SRG), a classical result of Neumaier \cite{Neumaier79} shows that a primitive SRG with the least eigenvalue $g\lambda_n$ has degree at most $O(\lambda_n^6)$  unless it belongs to two explicit infinite families; this was recently improved to $O(-\lambda_n^5)$ by Koolen--Lv--Markowsky--Park \cite{KLMP}. These proofs depend on the highly symmetric nature of SRGs and the fact that they have exactly $3$ distinct eigenvalues, so they do not translate immediately to arbitrary graphs.

In this paper, we prove \cref{conj:least-eigenvalue} up to a lower--order factor even when $G$ is not regular.
\begin{theorem}
    \label{thm:main}
    For any $\epsilon > 0$, the following holds for all $n \gg_{\epsilon} 1$.\footnote{This notation means that there exists a constant $n_0(\epsilon)$ depending only on $\epsilon$ such that the result holds for all $n \geq n_0(\epsilon)$.} If $G$ is a graph on $n$ vertices that is $\epsilon$-far from a disjoint union of cliques, then 
    $$\lambda_n \leq -n^{1/4 - \epsilon}.$$
\end{theorem}
The second and least eigenvalues of a regular graph are closely related to the \emph{positive discrepancy} and \emph{max--cut} respectively. In this paper, we mainly work with max--cut; analogous results for the positive discrepancy of regular graphs can be obtained by passing to the complement. Let $G$ be a graph with $n$ vertices and $m$ edges. The \emph{max--cut} of $G$, which we denote by $\mc(G)$, is the largest number edges in a bipartite subgraph of $G$. A simple probabilistic argument shows that $\mc(G) \geq \frac{m}{2}$, so we instead consider the \emph{surplus} defined as 
$$\surp(G) = \mc(G) - \frac{m}{2}.$$
A classical result of Edwards \cite{Edwards1, Edwards2} shows that $\surp(G) = \Omega(\sqrt{m})$. Since then, substantial efforts have been devoted to improve this bound for special graph families, in particular $H$-free graphs. Notably, Alon, Bollob\'{a}s, Krivelevich and Sudakov \cite{ABKS} conjectured that for any $H$, there exists a constant $\epsilon > 0$ such that any $K_r$-free graph $G$ satisfies
$$\surp(G) \geq c_H m^{3/4 + \epsilon}$$
where $c_H > 0$ is a constant depending on $H$ only. To prove this conjecture, we may assume that $H = K_r$. To our knowledge, the best previous result for $H = K_r$ was due to Glock, Janzer, and Sudakov \cite{GJS23}, who improved on a previous result of Carlson, Kolla, Li, Mani, Sudakov, and Trevisan \cite{CKLMST} and proved that any $K_r$-free graph $G$ satisfies $\surp(G) = \Omega(m^{\frac{1}{2} + \frac{3}{4r - 2}}).$ However, people did not even known whether the $3/4$ can be replaced with any absolute constant greater than $1/2$; Glock--Janzer--Sudakov \cite{GJS23} and Balla--Janzer--Sudakov \cite{BJS24} explicitly wondered if this holds. We refer to the survey \cite{WuHou} and the introduction of \cite{RST2023} for other partial results toward this conjecture.

A well-known result says that if $G$ is a $d$-degenerate graph, then we have
$$\surp(G) \geq \frac{m}{2\chi} \geq \frac{m}{2(d + 1)}$$
where $\chi$ is the chromatic number of $G$. For $K_r$-free graphs, Carlson, Kolla, Li, Mani, Sudakov, and Trevisan \cite{CKLMST} improved this to
$$\surp(G) \geq c_r \cdot \frac{m}{d^{1 - \frac{1}{2r - 4}}}.$$
Thus, the difficult case in the ABKS conjecture is when $d$ is close to $m^{1/2}$, which occurs if and only if the graph is dense (we use ``dense" to mean at least $n^{2-o(1)}$ edges, and ``very dense" to roughly mean $\Omega\left(n^2\right)$ edges). This motivates us to investigate the surplus of dense graphs in general, without the $H$-free assumption.

The expander--mixing lemma states that
\begin{equation}
\label{eq:expander-mixing}
 \surp(G) \leq - \lambda_n n.   
\end{equation}
So a lower bound on surplus implies a lower bound on least eigenvalue, but not vice versa. By adapting the method of Conlon and Zhao \cite{ConlonZhao}, we can show that \eqref{eq:expander-mixing} is tight up to a $\log n$ factor for vertex--transitive graphs. However, while some converse results of the expander mixing lemma hold for general graphs (e.g. \cite{ACHKRS}), we are unsure if \cref{thm:main} directly imply analogous statements for the surplus. Nevertheless, it hints that something can be said about graphs with surplus at most $n^{1 + c}$, where $c > 0$ is a small constant.

R\"{a}ty, Sudakov and Tomon \cite{RST2023} showed that for very dense regular graphs with degree $d \in \left[\left(\frac{1}{2} + \epsilon\right) n, \left(1 -  \epsilon\right) n\right]$, Edward's bound can be improved by a polynomial factor. 
$$\surp(G) \geq \Omega_\epsilon\left(\frac{n^{5/4}}{\log n}\right).$$
As a disjoint union of cliques has surplus $O(n)$, they \cite[Conjecture 8.1]{RST2023} conjectured that there is a similar polynomial improvement for graphs that are far from a disjoint union of cliques.\footnote{Their original formulation is in terms of the positive discrepancy $\disc^+(G)$. For regular graphs $G$, they showed that in the same paper that $\disc^+(G)$ and $\surp(\bar{G})$ are within a constant factor of each other. Thus, the statement we study here is equivalent.}
\begin{conjecture}
    \label{conj:discrepancy}
    For any $\epsilon > 0$, the following holds for some $c > 0$. If $G$ is a regular graph on $n$ vertices that is $\epsilon$-far from a disjoint union of cliques, then 
    $$\surp(G) \geq cn^{5/4}.$$
\end{conjecture}
R\"{a}ty, Sudakov and Tomon found it interesting to find any $\alpha > 0$ such that $\surp(G) \geq n^{1 + \alpha}$ is satisfied in \cref{conj:discrepancy}. By adapting the proof method of \cref{thm:main} with a number of technical modifications, we show that such an $\alpha$ exists even for irregular graphs.
\begin{theorem}
    \label{thm:discrepancy}
    Let $c = \frac{1}{100}$. For any $\epsilon > 0$, the following hold for all $n \gg_{\epsilon} 1$. If $G$ is a graph on $n$ vertices that is $\epsilon$-far from a disjoint union of cliques, then 
    $$\surp(G) \geq n^{1 + c}.$$
\end{theorem}
We define the \emph{edge density} of a graph $G$ as
$$p(G) = \frac{2m}{n^2}$$
where $m$ is the number of edges and $n$ is the number of vertices. \cref{thm:main} and \cref{thm:discrepancy} are relevant when the edge density of $G$ is $\Omega(1)$. However, for the ABKS conjecture, the relevant graphs have edge density $n^{-r(n)}$, where $r(n)$ goes to zero slowly with $n$. Our proofs of \cref{thm:main} and \cref{thm:discrepancy} only work when $n$ is at least exponential in $\epsilon^{-1}$, making them trivial if $G$ has at most $\frac{n^2}{\log n}$ edges. This motivates the following questions.
\begin{problem}
\label{prob:polynomial}
1)  Does \cref{thm:discrepancy} hold for $\epsilon = n^{-c}$, where $c > 0$ is some small constant?

2) More generally, does there exist some absolute constant $c > 0$ such that the following holds for any $\alpha > 1$ and $n > 0$? If $G$ is a graph on $n$ vertices with at least $n^{2 - c}$ edges and is $\alpha n^{-1}$-far from a disjoint union of cliques, then we must have
$$\surp(G) \geq c \cdot \alpha^{c} n.$$
\end{problem}
In this paper, we do not answer this question in full generality. Our next result makes some progress towards understanding the surplus of graphs with edge density in the range $[n^{-c}, 1]$. We show that if $G$ is a graph with at least $n^{1.999}$ edges and surplus at most $n^{1.001}$, then $G$ must contain a large, very dense subgraph. After a preliminary version of this manuscript was circulated, Jin--Milojevic--Tomon \cite{JMT2025} completely resolved part 1) of this problem by building on this theorem. We thank Tomon for sharing a draft of their paper.
\begin{theorem}
    \label{thm:polynomial}
    Let $\rho = \frac{1}{1000}$. For any $\epsilon > 0$, the following holds for all $n \gg_\epsilon 1$. Let $G$ be a graph on $n$ vertices with at least $n^{2 - \rho}$ edges. If $\surp(G) \leq n^{1 + \rho}$, then $G$ must contain an induced subgraph $H$ on at least $n^{1 - 4\rho}$ vertices with edge density at least $(1 - \epsilon)$.     
\end{theorem}
An immediate application is an improvement on the surplus of $K_r$-free graphs for large $r$.
\begin{corollary}
    \label{cor:ABKS?}
    The following holds for every $r$ and $m \gg_r 1$. Suppose $G$ is a $K_r$-free graph with $m$ edges. Let $\rho$ be the absolute constant in \cref{thm:polynomial}. Then we have
    $$\surp(G) \geq m^{\frac{1}{2} + \frac{\rho}{10}} = m^{0.5001}.$$
    In other words, there exists a constant $c_r > 0$ such that for every $K_r$-free graph $G$ with $m$ edges, we have
    $$\surp(G) \geq c_r m^{0.5001}.$$
\end{corollary}
Thus, we improve the exponent $\frac{1}{2}$ in Edwards' bound by an absolute constant for $H$-free graphs, answering the questions of Glock--Janzer--Sudakov \cite{GJS23} and Balla--Janzer--Sudakov \cite{BJS24}. We make no serious attempt to optimize the value of $\rho$ here, as getting close to $m^{\frac{3}{4}}$ certainly requires completely new ideas. 

As a final remark, Balla, Hambardzumyan and Tomon \cite{BHT25} recently showed a similar ``inverse theorem" for maxcut: if a graph $G$ satisfies $\surp(G) \leq \alpha m^{1/2}$, then $G$ contains a clique of size at least $2^{-O(\alpha^9)} \sqrt{m}$. Their result is complementary to ours, as \cref{thm:polynomial} cannot produce a clique of size $\Omega(m^{1/2})$, and their result is non--trivial only when $\alpha = O(\log^{1/9} m)$. A positive answer to \cref{prob:polynomial} would improve the clique size to $\alpha^{-C} \sqrt{m}$ for some absolute constant $C > 0$, giving us more motivation to study this problem.

The rest of the paper is organized as follows. In \cref{sec:overview} we introduce our notations and present three lemmas, Lemma~\ref{lem:recursive}, \ref{lem:top-concentration} and \ref{lem:spectrum-to-clique}, that constitute the proof of \cref{thm:main}. In \cref{sec:lem1}, \ref{sec:lem2}, \ref{sec:lem3} we prove each of the three lemmas, thus completing the proof of \cref{thm:main}. In \cref{sec:discrepancy-1} and \ref{sec:discrepancy-2} we prove \cref{thm:discrepancy} using a similar strategy, but with a key technical modification (\cref{lem:recursive-II}) to \cref{lem:recursive} described in \cref{sec:discrepancy-1}. Finally, by replacing \cref{lem:spectrum-to-clique} with a new argument in  \cref{lem:density-increment}, we prove \cref{thm:polynomial} in \cref{sec:ABKS?}.
\section*{Acknowledgement}
I am grateful to Anqi Li and Nitya Mani for the illuminating talks and discussions that introduced me to this subject. I thank Igor Balla and Istv\'{a}n Tomon for enlightening discussions about \cref{conj:second-eigenvalue} that motivate much of this work. My appreciation goes to my advisor, Jacob Fox, for guiding this project with insightful feedback and mentorship. I also thank Quanyu Tang and Clive Elphick for proposing a number of conjectures that inspired several key ideas developed here. Finally, I thank Jack Koolen and Qianqian Yang for helpful discussions regarding the Hoffman program. This material is partially based upon work supported by the National Science Foundation under Grant No. DMS-1928930, while the author was in residence at the Simons Laufer Mathematical Sciences Institute in Berkeley, California, during the Spring 2025 semester. This work was also partially supported by NSF Award DMS-2154129. {\color{white} Finally, I thank Koishi Komeiji for inspiring a key argument in this paper on May 14th, 2025.}


\section{Notations and Proof overview}
\label{sec:overview}
We begin by recalling some well-known linear algebraic notations and facts, which we will use throughout the paper. Then we give an overview of the proof strategy by formally stating the main intermediate results leading up to \cref{thm:main}.

All graphs are simple and undirected. Let $G$ be a graph with vertex set $V$ and edge set $E$, and let $n = |V|$ denote the number of vertices. For the sake of convenience, we always identify $V$ with $\{1,2,\cdots, n\}$. The \emph{adjacency matrix} of $G$ is the symmetric $n \times n$ matrix $A_G$, with rows and columns indexed by $V$, whose entry at $(i, j)$ is $1$ if $ij \in E$ and $0$ otherwise. As is standard in spectral graph theory, we order the eigenvalues of $A_G$ from largest and smallest, and denote them by
$$\lambda_1 \geq \cdots\geq \lambda_n.$$
The \emph{spectrum} of $G$ refers to these eigenvalues. Recall the classical identities
$$\sum_{i = 1}^n \lambda_i = 0,$$
$$\sum_{i = 1}^n \lambda_i^2 = 2 |E| \leq n^2.$$
Let $\bv_i \in \RR^n$ be a unit eigenvector of $A_G$ corresponding to $\lambda_i$, and write $\bv_{ij}$ as the entry of $\bv_i$ corresponding to vertex $j$. Recall the \emph{spectral decomposition}
$$A_G = \sum_{i = 1}^n \lambda_i  \bv_i \bv_i^\top.$$
For vectors $\bw$ and $\bw'$ in $\RR^n$, their \emph{entrywise product} $\bw \circ \bw'$ is the vector in $\RR^n$ with entries $(\bw \circ \bw')_i = \bw_i \bw'_i$. We analogously define the entrywise product of matrices. As $A_G$ is a $\{0, 1\}$-matrix, we note that
$$A_G = A_G \circ A_G.$$
The \emph{energy} of $G$ is defined as
$$\cE(G) = \sum_{i = 1}^n \abs{\lambda_i}.$$
Motivated by applications in molecular chemistry, the energy has attracted considerable attention from spectral graph theorists. We refer the reader to \cite{energy} for some classical discussions on the energy of a graph. We will not use many deep properties of energy in this paper; the relation we often use is simply
$$\cE(G) = 2\sum_{i : \lambda_i \geq 0} \lambda_i = 2\sum_{i : \lambda_i < 0} (-\lambda_i) \leq 2n(-\lambda_n).$$
For a threshold $T$, we define $L_T$ to be the set of $i$ such that $\lambda_i \geq T$, and set
$$S_T = \sum_{i \in L_T} \lambda_i.$$
For example, we have
$$S_0 = \sum_{i \in L_0} \lambda_i = \frac{1}{2} \cE(G).$$
Our main lemma is a recursive inequality bounding the positive eigenvalues of $G$. This observation will play crucial roles in all of our results.
\begin{lemma}
    \label{lem:recursive}
    For every graph $G$ on $n$ vertices and threshold $T \geq 4(-\lambda_n) \sqrt{n}$, we have
    $$S_T^2 \leq 2n S_{\frac{T^2}{4n}}.$$
\end{lemma}
We use an example to illustrate how this lemma can give non--trivial results. Assume $G$ satisfies $-\lambda_n \leq n^{1/4 - \epsilon}$. We can apply \cref{lem:recursive} with $T = n^{3/4}$, giving
$$S_{n^{3/4}} \leq \sqrt{2n \cE(G)} \leq 2n \sqrt{-\lambda_n} \leq 2n^{9/8 - \epsilon / 2}.$$
In contrast, the trivial $L^2$-bound 
$$S_T \leq \frac{\sum_{i = 1}^n \lambda_i^2}{T} \leq \frac{n^2}{T}$$
only gives $S_T \leq n^{5/4}$. Thus, \cref{lem:recursive} is power--saving, giving what we need to finish the proof.

As a remark, our proof of this lemma crucially depends on the fact that the adjacency matrix $A_G$ has $\{0, 1\}$-entries. This assumption is also crucially used in Balla--R\"{a}ty--Sudakov--Tomon's simple proof \cite{BRST} of their Alon--Boppana--type bound when $d \leq n / 2$.

Iterating \cref{lem:recursive} with a geometrically increasing threshold sequence $T_1, T_2, \cdots$, we arrive at
\begin{lemma}
\label{lem:top-concentration}
For any $\epsilon, \delta \in (0, 0.01)$, the following holds for all $n \gg_{\epsilon, \delta} 1$. For any graph $G$ on $n$ vertices with $\lambda_n \geq -n^{1/4 - \epsilon}$ we have
$$\sum_{i: \lambda_i \leq (\epsilon\delta)^{1 / \epsilon} n} \lambda_i^2 \leq \delta n^2.$$
\end{lemma}
This lemma reveals an important piece of information: the square mass of the eigenvalues of $G$ with magnitude $o(n)$ is $o(n^2)$. The next lemma translates the spectral information to the graph side, showing that such $G$ is either close to a disjoint union of cliques, or has a large negative eigenvalue.
\begin{lemma}
\label{lem:spectrum-to-clique}
For any $\epsilon > 0$, there exists a $\delta = \delta(\epsilon) > 0$ such that the following holds for all $\gamma > 0$. Let $G$ be a graph on $n \gg_{\gamma, \epsilon} 1$ vertices. Suppose $G$ satisfies
$$\sum_{i: \lambda_i < \gamma n} \lambda_i^2 \leq \delta n^2.$$
Then one of the followings hold.
\begin{enumerate}
    \item $G$ is $\epsilon$-close to a disjoint union of cliques.
    \item The least eigenvalue of $G$ is at most $-c_{\epsilon, \gamma} \cdot n$, where $c_{\epsilon, \gamma} > 0$ is a constant depending only on $\epsilon$ and $\gamma$.
\end{enumerate}
\end{lemma}
\cref{thm:main} readily follows from these lemmas.
\begin{proof}[Proof of \cref{thm:main}]
    We prove the contrapositive. Let $G$ be a graph on $n \gg_\epsilon 1$ vertices with $\lambda_n \geq -n^{1/4 - \epsilon}$. Let $\delta$ be the constant in \cref{lem:spectrum-to-clique}. Without loss of generality, we may assume that $\epsilon, \delta \in (0, 0.01)$. By \cref{lem:top-concentration}, we have
    $$\sum_{i: \lambda_i \leq (\epsilon\delta)^{1 / \epsilon} n} \lambda_i^2 \leq \delta n^2.$$
    Applying \cref{lem:spectrum-to-clique} with $\gamma = (\epsilon\delta)^{1 / \epsilon}$, we conclude that $G$ is either $\epsilon$-close to a disjoint union of cliques, or $\lambda_n(G) \leq -c_{\epsilon, \gamma} \cdot n$. The latter contradicts the assumption that $\lambda_n(G) \geq -n^{1/4 - \epsilon}$ for $n \gg_\epsilon 1$, so $G$ is $\epsilon$-close to a disjoint union of cliques as desired.
\end{proof}
\section{The key recursion}
\label{sec:lem1}
As motivation for \cref{lem:recursive}, we first consider the case where $G = \mathsf{Cay}(A, S)$ is the Cayley graph of an abelian group $A$. Let $\hat{A}$ be the dual group of $A$. The spectrum of $G$ is given by the Fourier transforms of $S$
$$\lambda_i = \hat{1}_S(\chi_i) = \sum_{a \in S} \chi_i(a)$$
where $\chi_1, \cdots, \chi_n$ are distinct elements of $\hat{A}$. We present a short proof of Lemma~\ref{lem:recursive} in this setting. 
\begin{proof}[Proof of Lemma~\ref{lem:recursive} for Abelian Cayley graphs]
    Observe that $1_S = 1_S^2$. Taking the Fourier transform, for any $\chi \in \hat{A}$ we have
    $$\hat{1}_S(\chi) = \frac{1}{n}\sum_{\xi + \mu = \chi} \hat{1}_S(\xi) \hat{1}_S(\mu).$$
    
    Let $L_T + L_T$ denote the set of $i$ such that $\chi_i = \chi_j + \chi_k$ for some $j, k \in L_T$. For each $i \in L_T + L_T$, we have
    $$\hat{1}_S(\chi_i) = \frac{1}{n} \sum_{\xi + \mu = \chi_i} \hat{1}_S(\xi) \hat{1}_S(\mu).$$
    Dropping non-negative terms on the right hand side, we obtain
    $$\hat{1}_S(\chi_i) \geq \frac{1}{n} \sum_{\substack{\chi_j + \chi_k = \chi_i \\ j, k \in L_{T}}} \hat{1}_S(\chi_j) \hat{1}_S(\chi_k) +  \frac{2}{n}\sum_{\substack{\xi + \mu = \chi_i \\ \hat{1}_S(\xi) > 0, \hat{1}_S(\mu) < 0}} \hat{1}_S(\xi) \hat{1}_S(\mu).$$
    Using the fact that $\hat{1}_S(\mu) \geq -\lambda_n$, we get
    $$\hat{1}_S(\chi_i) \geq \frac{1}{n} \sum_{\substack{\chi_j + \chi_k = \chi_i \\ j, k \in L_{T}}} \hat{1}_S(\chi_j) \hat{1}_S(\chi_k) - \frac{2 (-\lambda_n)}{n} \sum_{\xi: \hat{1}_S(\xi) > 0} \hat{1}_S(\xi) = \frac{1}{n}\sum_{\substack{\chi_j + \chi_k = \chi_i \\ j, k \in L_{T}}} \hat{1}_S(\chi_j) \hat{1}_S(\chi_k) - \frac{-\lambda_n}{n}\cE(G).$$
    Summing over all $i \in L_T + L_T$, we have
    $$\sum_{i \in L_T + L_T} \hat{1}_S(\chi_i) \geq \frac{1}{n} \sum_{i \in L_T + L_T}\sum_{\substack{\chi_j + \chi_k = \chi_i \\ j, k \in L_{T}}} \hat{1}_S(\chi_j) \hat{1}_S(\chi_k) - \frac{-\lambda_n}{n} \cE(G) \abs{L_T + L_T}.$$
    For the left hand side, we split the sum based on whether $i \in L_{\frac{T^2}{4n}}$ to obtain
    $$\sum_{i \in L_T + L_T} \hat{1}_S(\chi_i) \leq \sum_{i \in L_{\frac{T^2}{4n}}} \hat{1}_S(\chi_i) + \sum_{\substack{i \in L_T + L_T \\ i \notin L_{\frac{T^2}{4n}}}} \hat{1}_S(\chi_i) \leq S_{\frac{T^2}{4n}} + \frac{T^2}{4n} \abs{L_T + L_T}.$$
    For the right hand side, note that 
    $$\sum_{i \in L_T + L_T}\sum_{\substack{\chi_j + \chi_k = \chi_i \\ j, k \in L_{T}}} \hat{1}_S(\chi_j) \hat{1}_S(\chi_k) = S_T^2.$$
    So we conclude that
    $$S_{\frac{T^2}{4n}} + \frac{T^2}{4n} \abs{L_T + L_T} \geq \frac{1}{n}  S_T^2 - \frac{-\lambda_n}{n}\cE(G) \abs{L_T + L_T}.$$
    Finally, using the facts that 
    $$\abs{L_T + L_T} \leq \abs{L_T}^2 \leq \frac{S_T^2}{T^2} \quad \text{ and } \cE(G) \leq 2n (-\lambda_n)$$
    the above simplifies to
    $$n S_{\frac{T^2}{4n}} + \frac{1}{4} S_T^2 \geq S_T^2 - \frac{2n (-\lambda_n)^2}{T^2} S_T^2.$$
    By the assumption that $T \geq 4(-\lambda_n) \sqrt{n}$, we obtain the inequality
    $$n S_{\frac{T^2}{4n}} + \frac{1}{4} S_T^2 \geq S_T^2 - \frac{1}{4} S_T^2$$
    giving
    $$S_T^2 \leq 2n S_{\frac{T^2}{4n}}$$
    as desired.
\end{proof}
In order to mirror the proof for general graphs $G$, we begin with the identity
$$A_G = A_G \circ A_G.$$
Applying the spectral decomposition to both sides, we have
$$\sum_{i = 1}^n \lambda_i  \bv_i \bv_i^\top = \sum_{i, j = 1}^n \lambda_i \lambda_j (\bv_i \bv_i^\top) \circ (\bv_j \bv_j^\top) = \sum_{i, j = 1}^n \lambda_i \lambda_j (\bv_i \circ \bv_j) (\bv_i \circ \bv_j)^\top.$$
Therefore, for each test vector $q \in \RR^n$, we have
$$q^\top \left(\sum_{i = 1}^n \lambda_i   \bv_i \bv_i^\top\right) q  = q^\top \left(\sum_{i, j = 1}^n \lambda_i \lambda_j (\bv_i \circ \bv_j) (\bv_i \circ \bv_j)^\top\right)q.$$
Simplifying, we obtain
$$\sum_{i = 1}^n \lambda_i \langle \bv_i, q\rangle^2 = \sum_{i, j = 1}^n \lambda_i \lambda_j \langle \bv_i \circ \bv_j, q\rangle^2.$$
It remains to choose a good test vector $q$. Motivated by the proof of the Abelian Cayley case, we choose $q$ to be a random Gaussian vector on a certain subspace.\footnote{Equivalently, here we can consider the quantity $\tr(\Pi_W A_G \Pi_W)$, where $\Pi_W$ is a projection matrix defined below. }
\begin{definition}
    Let $W$ be a $d$-dimensional linear subspace of $\RR^n$. A \emph{standard Gaussian vector $q \sim N(0, I_W)$ on W} is a random variable defined as follows: take an orthonormal basis $\bw_1, \cdots, \bw_d$ of $W$, let $x \sim N(0, I_d)$ be the standard $d$-dimensional Gaussian random variable, and let
    $$q = \sum_{i = 1}^d x_i \bw_i.$$
    By the rotational invariance of the Gaussian random variable, $q$ is independent of the choice of the basis $\bw_i$.
\end{definition}
 We collect some useful properties of $q$ in the following proposition, which can be easily verified.
\begin{proposition}
\label{prop:Gaussian-basic}
    Let $q \sim N(0, I_W)$. Then we have
    
    1) We have $\EE \norm{q}_2^2 = \dim W$.

    2) For any $\bw \in \RR^n$, we have
    $$\EE \langle \bw, q \rangle^2 = \norm{\Pi_W \bw}_2^2 \leq \norm{\bw}_2^2.$$
    where $\Pi_W$ is the orthogonal projection onto $q$. \qed
\end{proposition}
Now let us prove Lemma~\ref{lem:recursive} in full generality.
\begin{proof}[Proof of Lemma~\ref{lem:recursive}]
Let $W$ be the subspace of $\RR^n$ spanned by the vectors $\{\bv_i \circ \bv_j: i, j \in S_T\}$. Let $q\sim N(0, I_W)$ be the standard Gaussian vector on $W$, and consider the identity 
$$\sum_{i = 1}^n \lambda_i \langle \bv_i, q\rangle^2 = \sum_{i, j = 1}^n \lambda_i \lambda_j \langle \bv_i \circ \bv_j, q\rangle^2.$$
We estimate the expectations of both sides. For the left--hand side, Parseval's identity gives
$$\sum_{i \notin L_{\frac{T^2}{4n}}}\lambda_i \langle \bv_i, q\rangle^2 \leq \frac{T^2}{4n} \sum_{i = 1}^n \langle \bv_i, q\rangle^2 = \frac{T^2}{4n} \norm{q}_2^2.$$
Therefore, we have
$$\sum_{i = 1}^n \lambda_i \langle \bv_i, q\rangle^2 \leq \sum_{i \in L_{\frac{T^2}{4n}}} \lambda_i \langle \bv_i, q\rangle^2 + \frac{T^2}{4n} \norm{q}_2^2.$$
Taking expectation, we have
$$\EE \sum_{i = 1}^n \lambda_i \langle \bv_i, q\rangle^2 \leq \sum_{i \in L_{\frac{T^2}{4n}}} \lambda_i \EE \langle \bv_i, q\rangle^2 + \frac{T^2}{4n} \EE \norm{q}_2^2$$
In light of \cref{prop:Gaussian-basic}, we obtain
\begin{equation}
    \label{eq:LHS}
    \EE \sum_{i = 1}^n \lambda_i \langle \bv_i, q\rangle^2 \leq S_{\frac{T^2}{4n}} + \frac{T^2}{4n} \dim W.
\end{equation}
We move on to the right--hand side. Dropping non-negative terms, we have
$$\sum_{i, j = 1}^n \lambda_i \lambda_j \langle \bv_i \circ \bv_j, q\rangle^2 \geq \sum_{i, j \in L_T} \lambda_i \lambda_j \langle \bv_i \circ \bv_j, q\rangle^2 - 2\sum_{i \in L_0, j\notin L_0} \lambda_i (-\lambda_j) \langle \bv_i \circ \bv_j, q\rangle^2.$$
For $i, j \in L_T$, we have $\bv_i \circ \bv_j \in W$ by the definition of $W$. By \cref{prop:Gaussian-basic} (2) we have
$$\EE \langle \bv_i \circ \bv_j, q\rangle^2 = \norm{\bv_i \circ \bv_j}^2 = \sum_{k = 1}^n \bv_{ik}^2 \bv_{jk}^2.$$
Therefore, we obtain
$$\EE \sum_{i, j \in L_T} \lambda_i \lambda_j \langle \bv_i \circ \bv_j, q\rangle^2 = \sum_{i, j \in L_T} \sum_{k = 1}^n \lambda_i \lambda_j \bv_{ik}^2 \bv_{jk}^2 = \sum_{k = 1}^n \left(\sum_{i \in L_T} \lambda_i \bv_{ik}^2\right)^2.$$
Applying the Cauchy-Schwarz inequality, we conclude
\begin{equation}
    \label{eq:RHS-1}
\EE \sum_{i, j \in L_T} \lambda_i \lambda_j \langle \bv_i \circ \bv_j, q\rangle^2 \geq \frac{1}{n} \left(\sum_{k = 1}^n \sum_{i \in L_T} \lambda_i \bv_{ik}^2\right)^2 = \frac{1}{n} S_T^2.
\end{equation}
Finally, note that
$$\sum_{i \in L_0, j\notin L_0} \lambda_i (-\lambda_j) \langle \bv_i \circ \bv_j, q\rangle^2 \leq (-\lambda_n) \sum_{i \in L_0, j\notin L_0} \lambda_i \langle \bv_i \circ \bv_j, q\rangle^2.$$
Using the identity $\langle \bv_i \circ \bv_j, q\rangle = \langle \bv_i \circ q, \bv_j \rangle$ and Parseval's identity, we have
$$\sum_{j \notin L_0} \langle \bv_i \circ \bv_j, q\rangle^2 \leq \norm{\bv_i \circ q}^2.$$
Therefore, we have
$$\sum_{i \in L_0, j\notin L_0} \lambda_i (-\lambda_j) \langle \bv_i \circ \bv_j, q\rangle^2 \leq (-\lambda_n) \sum_{i \in L_0} \lambda_i \norm{\bv_i \circ q}^2 = (-\lambda_n) \sum_{i \in L_0} \sum_{k = 1}^n \lambda_i \bv_{ik}^2 q_k^2.$$
We now consider the spectral decomposition again
$$A_G = \sum_{i = 1}^n \lambda_i  \bv_i \bv_i^\top.$$
Considering the diagonal entry at $(k, k)$, we have
$$0 = \sum_{i = 1}^n \lambda_i  \bv_{ik}^2.$$
Therefore, we have
$$\sum_{i \in L_0} \lambda_i \bv_{ik}^2 = \sum_{i \notin L_0} (-\lambda_i) \bv_{ik}^2 \leq -\lambda_n \sum_{i \notin L_0}\bv_{ik}^2 \leq -\lambda_n.$$
So we get
$$\sum_{i \in L_0, j\notin L_0} \lambda_i (-\lambda_j) \langle \bv_i \circ \bv_j, q\rangle^2 \leq (-\lambda_n) \sum_{k = 1}^n (-\lambda_n) q_k^2 = (-\lambda_n)^2 \norm{q}_2^2.$$
By \cref{prop:Gaussian-basic}, we obtain
\begin{equation}
\label{eq:RHS-2}
\EE \sum_{i \in L_0, j\notin L_0} \lambda_i (-\lambda_j) \langle \bv_i \circ \bv_j, q\rangle^2 \leq (-\lambda_n)^2  \dim W.
\end{equation}
Finally, combining \eqref{eq:LHS}, \eqref{eq:RHS-1} and \eqref{eq:RHS-2}, we obtain
$$S_{\frac{T^2}{4n}} + \frac{T^2}{4n} \dim W \geq \frac{1}{n}S_T^2 - 2(-\lambda_n)^2  \dim W.$$
By definition, we have $\dim W \leq \abs{L_T}^2 \leq \frac{S_T^2}{T^2}$. Substituting into the above, we have
$$ nS_{\frac{T^2}{4n}} + \frac{1}{4} S_T^2 \geq S_T^2 - \frac{2n(-\lambda_n)^2}{T^2} S_T^2.$$
Using the assumption that $T \geq 4\sqrt{n} (-\lambda_n)$, we conclude that
$$nS_{\frac{T^2}{4n}} \geq \frac{S_T^2}{2}$$
as desired.
\end{proof}
\section{Solving the recursion}
\label{sec:lem2}
In this section, we prove \cref{lem:top-concentration} by solving the recursion given by \cref{lem:recursive}. As we have to repeat this argument later with different parameters, we show a general lemma for solving such recursions given appropriate energy and least eigenvalue bounds.
\begin{lemma}
\label{lem:solving-recursion-general}
Let $p,q,r, C > 0$ be constants, and $n \gg_{p,q,r,C} 1$. Let $\lambda_1 \geq  \cdots \geq \lambda_n$ be a sequence of real numbers. Let $L_T = \{i \in [n]: \lambda_i \geq T\}$, and let $S_T = \sum_{i \in L_T} \lambda_i$; note that $S_T$ is non--increasing in $T$ when $T \geq 0$. Suppose that
\begin{enumerate}
    \item We have $-\lambda_n \leq Cn^p.$

    \item We have $\sum_{i = 1}^n \abs{\lambda_i} \leq Cn^q$.
    
    \item $S_T$ satisfies the following recursive relation for all $T \geq Cn^r$
    $$S_T^2 \leq Cn S_{\frac{T^2}{Cn}}.$$
\end{enumerate}
and the following relations hold between the constant parameters
\begin{enumerate}
    \item $p,r \in (0, 1)$, $q \in (1, 2)$, $C \geq 1$.
    \item $q + \max(p, r) < 2$.
\end{enumerate}
Then for any $H \in [n^{p + q - 1}, n]$, we have
$$\sum_{i \notin L_H} \lambda_i^2 \leq \frac{2C^{1 + s}}{1 - s} \cdot n^{1 + s} H^{1 - s}$$
where $s = \frac{q - 1}{1 - r}$.
\end{lemma}
\begin{proof}
Let $T_0 = Cn^{2r - 1}$, and recursively define $T_{i + 1} = (CnT_i)^{1/2}$ for $i \geq 0$. Then we have
$$\frac{Cn}{T_{i + 1}} = \sqrt{\frac{Cn}{T_i}}.$$   
As $T_1 = Cn^r$ and $T_i$ is non-decreasing, we have $T_{i + 1} \geq Cn^r$ for all $i \geq 0$, so condition 3) shows that
$$S_{T_{i + 1}} \leq \sqrt{Cn S_{T_i}}.$$
We now argue by induction that for each $i \geq 0$
\begin{equation}
    \label{eq:induction}
\quad S_{T_{i}} \leq Cn \cdot \left(\frac{Cn}{T_{i + 1}}\right)^{s}.
\end{equation} 
For the base case $i = 0$, note that 
$$S_{T_0} \leq \sum_{i = 1}^n \abs{\lambda_i} \leq Cn^q$$
and
$$Cn \cdot \left(\frac{Cn}{T_{1}}\right)^{s} = Cn \cdot \left(\frac{Cn}{Cn^r}\right)^{s} = Cn^{1 + (1 - r)s} = Cn^{q}.$$
so the base case holds.

Now suppose \eqref{eq:induction} holds for some $i \geq 0$. We have
$$S_{T_{i + 1}} \leq \sqrt{Cn S_{T_i}} \leq \sqrt{Cn \cdot Cn \cdot \left(\frac{Cn}{T_{i + 1}}\right)^{s}} = Cn \cdot \left(\sqrt{\frac{Cn}{T_{i + 1}}}\right)^{s} = Cn \cdot \left(\frac{Cn}{T_{i + 2}}\right)^{s}.$$
So \eqref{eq:induction} holds for $(i + 1)$ as well. By induction, the desired inequality \eqref{eq:induction} holds for all $i$. 

Observe that $\lim_{i \to \infty} T_i = Cn \geq n$. Hence, for each $T \in [T_0, H]$, there exists some $i \geq 0$ such that $T_i \leq T \leq T_{i + 1}$. As $S_{T}$ is non--increasing in $T$, we have
$$S_T \leq S_{T_i} \leq Cn \cdot \left(\frac{Cn}{T_{i + 1}}\right)^{s} \leq Cn \cdot \left(\frac{Cn}{T}\right)^{s}.$$
For $T \in [0, T_0]$, we also have
$$S_T \leq \sum_{i = 1}^n \abs{\lambda_i} \leq Cn^q = Cn \cdot \left(n^{1 - r}\right)^s = Cn \cdot \left(\frac{Cn}{T_1}\right)^{s} \leq Cn \cdot \left(\frac{Cn}{T}\right)^{s}.$$
Hence, we conclude that
$$\sum_{i: \lambda_i \in [0, H]} \lambda_i^2 \leq \int_0^{H} S_t dt \leq \int_{0}^{H} Cn \cdot \left(\frac{Cn}{t}\right)^{s} dt.$$
As $q + r < 2$, we have $s = \frac{q - 1}{1 - r} < 1$, so the integral converges, and we obtain
$$\sum_{i: \lambda_i \in [0, H]} \lambda_i^2 \leq \frac{1}{1 - s} \cdot (Cn)^{1 + s} H^{1 - s}.$$
It remains to estimate the contribution of the negative $\lambda_i$'s, which we achieve through 1). We have 
$$\sum_{i \notin L_{0}} \lambda_i^2 \leq \max(-\lambda_n, 0) \cdot \sum_{i = 1}^n \abs{\lambda_i} \leq C^2 n^{p + q}.$$
As $H \geq n^{p + q - 1}$, both $(2 - p - q)$ and $s$ are positive, and $n$ is sufficiently large, we have
$$n^{1 + s} H^{1 - s} \geq n^{p + q} n^{(2 - p - q)s} \geq C^2 n^{p + q}.$$
We conclude that
$$\sum_{i \notin L_H} \lambda_i^2 = \sum_{i: \lambda_i \in [0, H]} \lambda_i^2 + \sum_{i \notin L_{0}} \lambda_i^2 \leq  \frac{2C^{1 + s}}{1 - s}n^{1 + s}H^{1 - s}$$
as desired.
\end{proof}
\begin{proof}[Proof of \cref{lem:top-concentration}]
We check that the spectrum of $G$ satisfies all three conditions of \cref{lem:solving-recursion-general}. 

1) By assumption, we have $-\lambda_n \leq n^{1/4 - \epsilon}$.

2) We have
$$\sum_{i = 1}^n \abs{\lambda_i} = \cE(G) \leq 2n(-\lambda_n) \leq 2n^{5/4 - \epsilon}.$$

3) By \cref{lem:recursive}, for all $T \geq 4(-\lambda_n) \sqrt{n} = 4n^{3/4 - \epsilon}$, we have
$$S_T^2 \leq 2n S_{\frac{T^2}{4n}}.$$
Therefore, \cref{lem:solving-recursion-general} applies with the choice of parameters
$$(p,q,r,C) = \left(\frac{1}{4} - \epsilon, \frac{5}{4} - \epsilon, \frac{3}{4} - \epsilon, 4\right).$$
Take $H = (\epsilon\delta)^{1 / \epsilon} n$. Note that $H \geq n^{p + q - 1}$ for $n \gg_{\delta, \epsilon} 1$, and $s = \frac{q - 1}{1 - r} = \frac{\frac{1}{4} - \epsilon}{\frac{1}{4} + \epsilon} \leq 1 - 4\epsilon$. Thus we have
$$\sum_{i \notin L_H} \lambda_i^2 \leq \frac{2\cdot 4^{2 - 4\epsilon}}{4\epsilon} n^{2 - 4\epsilon} H^{4\epsilon}.$$
Simplifying, we obtain
$$\sum_{i \notin L_H} \lambda_i^2 \leq \frac{2\cdot 4^{2 - 4\epsilon}}{4\epsilon} (\epsilon\delta )^4 n^{2} \leq \delta n^2$$
as desired.
\end{proof}
\section{Translating the spectral information}
\label{sec:lem3}
In this section, we prove \cref{lem:spectrum-to-clique}, which translates the spectral bound in \cref{lem:top-concentration} into structural information about the graph $G$. First, we show a criterion for a graph to be close to a disjoint union of cliques. Recall that a \emph{cherry} is an induced copy of $K_{1, 2}$.
\begin{lemma}
\label{lem:cherries}
   For every $\epsilon > 0$, there exists some constant $\delta_{K_{1, 2}} = \delta_{K_{1, 2}}(\epsilon) > 0$ such that the following holds. Suppose a graph $G$ on $n$ vertices has at most $\delta_{K_{1, 2}} n^3$ cherries. Then $G$ is $\epsilon$-close to a disjoint union of cliques. 
\end{lemma}
\begin{proof}
    As a graph is a disjoint union of cliques if and only if it has no induced copy of $K_{1, 2}$, this claim is equivalent to the induced graph removal lemma \cite{AFKS2000} with $H = K_{1, 2}$.
\end{proof}
We take $\delta = \left(\frac{\delta_{K_{1, 2}}(\epsilon)}{12}\right)^3$. Without loss of generality, we also assume that $\delta \leq \frac{1}{1000}$ and $\gamma \leq \delta$. Assume $G = (V, E)$ satisfies the assumption of \cref{lem:spectrum-to-clique}, i.e.
$$\sum_{i: \lambda_i < \gamma n} \lambda_i^2 \leq \delta n^2.$$

We perform spectral partitioning on $G$. Let $S = L_{\gamma n}$ be the set of $i$ such that $\lambda_i \geq \gamma n$. Note that
$$|S| \leq \frac{\sum_{i = 1}^n \lambda_i^2}{(\gamma n)^2} \leq \gamma^{-2}.$$
For each vertex $v$, we associate the $|S|$-dimensional vector $H_v \in \RR^S$ whose $i$-th entry is defined by
$$H_{v, i} = \sqrt{\lambda_i} \bv_{iv}.$$
Let $H$ denote the $V \times S$ matrix whose $v$-th row is $H_v$. Then we have
$$HH^\top = \sum_{i \in S} \lambda_i \bv_i \bv_i^\top.$$
Therefore, the spectrum of the matrix $HH^\top - A_G$ is $\{\lambda_i: i \notin S\}$ together with $|S|$ copies of $0$. Thus we have
$$\sum_{u, v \in V} (HH^\top - A_G)_{uv}^2 = \sum_{i \notin S} \lambda_i^2 \leq \delta n^2.$$
We now partition $V(G)$ based on the vectors $H_v$. First, let's estimate the magnitude of such vectors. For each $i \in S$ and $v \in V$, we have
$$H_{v, i} = \sqrt{\lambda_i} \bv_{iv} = \frac{1}{\sqrt{\lambda_i}} \sum_{j \sim i} \bv_{jv}\leq \frac{\sqrt{n}}{\sqrt{\lambda_i}} \leq \gamma^{-1/2}.$$
Therefore, we have $\norm{H_v}_\infty \leq \gamma^{-1/2}$ for every $v \in V$. Thus, we can paritition $V$ into
$$t = \ceil{2\gamma^{-10} + 1}^{\gamma^{-2}}$$ 
parts $V_1, \cdots, V_t$, such that for two vertices $u, v$ in the same part, we have
$$\norm{H_u - H_v}_\infty \leq \gamma^{9.5}.$$
Thus, for any four vertices $u, u', v, v'$ such that $u$ and $u'$, $v$ and $v'$ lie in the same part, we have
$$\abs{\langle H_u, H_v \rangle - \langle H_{u'}, H_{v'} \rangle} \leq \norm{H_u - H_{u'}}_\infty \norm{H_v}_1 + \norm{H_v - H_{v'}}_\infty \norm{H_{u'}}_1 \leq 2|S| \gamma^{9.5} \cdot \gamma^{-1/2} \leq 2 \gamma^7.$$
We conclude that, for each pair of parts $V_i$ and $V_j$, there is a constant $a_{ij} \in \RR$ such that 
$$\abs{\langle H_u, H_v \rangle - a_{ij}} \leq \gamma^7$$
for every $(u, v) \in V_i \times V_j$.

Let $e(V_i, V_j)$ denote the number of pairs $(u, v) \in V_i \times V_j$ such that $uv$ is an edge in $G$. Following the notation in \cite{NSS2024}, for each positive real number $\mu > 0$, we say a pair of parts $V_i$ and $V_j$ is \emph{$\mu$-dense} if $e(V_i, V_j) \geq (1 - \mu) |V_i| |V_j|$, \emph{$\mu$-sparse} if $e(V_i, V_j) \leq \mu |V_i| |V_j|$, and \emph{$\mu$-impure} if neither holds. We build towards a proof of \cref{lem:spectrum-to-clique} by proving a series of claims.

\textbf{Claim 1: }For any $\mu < 0.1$, if there are three (not necessarily distinct) parts $V_i, V_j, V_k$ such that $V_i$ and $V_j$ are $\mu$-dense, $V_i$ and $V_k$ are $\mu$-dense, and $V_j$ and $V_k$ are $\mu$-sparse, then we have
$$\lambda_n \leq -\frac{\min(|V_i|, |V_j|, |V_k|)}{10}.$$
\begin{proof}[Proof of Claim 1]
Clearly, we have $i \neq j$ and $i \neq k$. We divide into two cases based on whether $j = k$.

\textbf{Case 1: }Suppose $j \neq k$. We consider the test vector $\bx \in \RR^V$ given by 
$$\bx_v = \begin{cases}
\frac{1}{|V_i|}, v \in V_i \\
-\frac{1}{|V_j|}, v \in V_j \\
-\frac{1}{|V_k|}, v \in V_k \\
0, \text{otherwise}
\end{cases}.$$
We compute that
$$\bx^\top A_G \bx = \sum_{t \in \{i,j,k\}} \frac{e(V_t, V_t)}{|V_t|^2} - 2 \frac{e(V_i, V_j)}{|V_i| |V_j|} - 2 \frac{e(V_i, V_k)}{|V_i| |V_k|} + 2\frac{e(V_j, V_k)}{|V_j| |V_k|}.$$
By the trivial bound $e(V_t, V_t) \leq |V_t|^2$ and the assumptions, we obtain
$$\bx^\top A_G \bx \leq 3 - 2(1 - \mu) - 2(1 - \mu) + 2 \mu \leq -1 + 6 \mu \leq -0.4.$$
On the other hand, we have
$$\bx^\top \bx = \frac{1}{|V_i|} + \frac{1}{|V_j|} + \frac{1}{|V_k|} \leq \frac{3}{\min(|V_i|, |V_j|, |V_k|)}.$$
Computing the Rayley quotient, we conclude that
$$\lambda_n(A_G) \leq \frac{\bx^\top A_G \bx}{\bx^\top \bx} \leq -\frac{\min(|V_i|, |V_j|, |V_k|)}{10}$$
as desired.

\textbf{Case 2: }Suppose $j = k$. Analogously, we consider the test vector $\bx \in \RR^V$ given by 
$$\bx_v = \begin{cases}
\frac{1}{|V_i|}, v \in V_i \\
-\frac{1}{|V_j|}, v \in V_j \\
0, \text{otherwise}
\end{cases}.$$
Then we can compute that
$$\bx^\top A_G \bx = \frac{e(V_i, V_i)}{|V_i|^2} - 2 \frac{e(V_i, V_j)}{|V_i| |V_j|} + \frac{e(V_j, V_j)}{|V_j|^2}.$$
By the trivial bound $e(V_i, V_i) \leq |V_i|^2$ and the assumptions, we obtain
$$\bx^\top A_G \bx \leq 1 - 2(1 - \mu) + \mu \leq -1 + 3\mu \leq 0.4.$$
On the other hand, we have
$$\bx^\top \bx = \frac{1}{|V_i|} + \frac{1}{|V_j|} \leq \frac{2}{\min(|V_i|, |V_j|, |V_k|)}.$$
Computing the Rayley quotient, we conclude that
$$\lambda_n \leq \frac{\bx^\top A_G \bx}{\bx^\top \bx} \leq -\frac{\min(|V_i|, |V_j|, |V_k|)}{10}$$
as desired.
\end{proof}
\textbf{Claim 2: }We have
$$\sum_{(V_i, V_j) \text{ is $\delta^{1/3}$-impure}} |V_i||V_j| \leq 2\delta^{1/3} n^2.$$
\begin{proof}[Proof of Claim 2]
    Suppose $(V_i, V_j)$ is $\delta^{1/3}$-impure. Recall that there is a constant $a_{ij} \in \RR$ such that
    $$\abs{\langle H_u, H_v \rangle - a_{ij}} \leq \gamma^7$$
    for all $(u, v) \in V_i \times V_j$. Therefore, we have
    $$\sum_{u \in V_i, v \in V_j} \abs{\langle H_u, H_v \rangle - (A_G)_{uv}} \geq \abs{a_{ij}} (|V_i||V_j| - e(V_i, V_j)) + \abs{1 - a_{ij}} e(V_i, V_j) - \gamma^7 |V_i||V_j|.$$
    As $\abs{a_{ij}} + \abs{1 - a_{ij}} \geq 1$, we obtain
    $$\sum_{u \in V_i, v \in V_j} \abs{\langle H_u, H_v \rangle - (A_G)_{uv}} \geq \min(|V_i||V_j| - e(V_i, V_j), e(V_i, V_j)) - \gamma^7 |V_i||V_j|.$$
    By the definition of impurity, we obtain
    $$\sum_{u \in V_i, v \in V_j} \abs{\langle H_u, H_v \rangle - (A_G)_{uv}} \geq (\delta^{1/3} - \gamma^7) |V_i||V_j|.$$
    Applying the Cauchy-Schwarz inequality and the assumption $\gamma \leq \delta$, we have
    $$\sum_{u \in V_i, v \in V_j} \abs{\langle H_u, H_v \rangle - (A_G)_{uv}}^2 \geq (\delta^{1/3} - \gamma^7)^2 |V_i||V_j| \geq \frac{1}{2} \delta^{2/3} |V_i||V_j|.$$
    On the other hand, recall that
    $$\sum_{u, v \in V} (HH^\top - A_G)_{uv}^2 \leq \delta n^2.$$
    This implies that
    $$\sum_{(V_i, V_j) \text{ is $\delta^{1/3}$-impure}} \frac{1}{2}\delta^{2/3}|V_i||V_j| \leq \sum_{u \in V, v \in V} \abs{\langle H_u, H_v \rangle - (A_G)_{uv}}^2 \leq \delta n^2.$$
    So we have
    $$\sum_{(V_i, V_j) \text{ is $\delta^{1/3}$-impure}} |V_i||V_j| \leq 2\delta^{1/3} n^2$$
    as desired.
\end{proof}
Motivated by the preceding claims, we say a pair of vertices $(u, v) \in V \times V$ is \emph{bad} if any of the following holds. Let $V_i$ be the part $u$ lies in, and $V_j$ be the part $v$ lies in. Recall that $t$ denotes the number of parts $V_i$.
\begin{enumerate}
    \item We have $\min(|V_i|, |V_j|) \leq \delta t^{-1} n$.
    \item The pair $(V_i, V_j)$ is $\delta^{1/3}$-impure.
    \item The pair $(V_i, V_j)$ is $\delta^{1/3}$-dense and $uv$ is not an edge, or the pair $(V_i, V_j)$ is $\delta^{1/3}$-sparse and $uv$ is an edge.
\end{enumerate}
\textbf{Claim 3: }The number of bad vertex pairs in $G$ is at most $4\delta^{1/3} n^2$.
\begin{proof}[Proof of Claim 3]
The number of vertex pairs satisfying 1) is at most
$$2 \sum_{i: |V_i| \leq \delta t^{-1} n} \sum_{j = 1}^t |V_i| |V_j| \leq 2n \sum_{i: |V_i| \leq \delta t^{-1} n} |V_i| \leq 2\delta n^2.$$
By Claim 2, the number of vertex pairs satisfying 2) is at most
$$\sum_{(V_i, V_j) \text{ is $\delta^{1/3}$-impure}} |V_i||V_j| \leq 2\delta^{1/3} n^2.$$
For each $(V_i, V_j)$, the number of pairs $(u, v) \in V_i \times V_j$ satisfying 3) is at most $\delta^{1/3} |V_i||V_j|$. Therefore, the total number of vertex pairs satisfying 3) is at most
$$\sum_{i, j = 1}^t \delta^{1/3} |V_i||V_j| \leq \delta^{1/3} n^2.$$
Summing over all three cases, the total number of bad vertex pairs is at most
$$2\delta n^2 + 2\delta^{1/3} n^2 + \delta^{1/3} n^2 \leq 4 \delta^{1/3} n^2$$
as desired.
\end{proof}
We are ready to complete the proof of \cref{lem:spectrum-to-clique}. We divide into two cases.

\textbf{Case 1: }Suppose $G$ contains a cherry on vertices $(u,v, w)$, with $uv, uw$ edges and $uw$ not an edge, such that $(u, v), (v, w), (w, u)$ are all good pairs. Let $V_i ,V_j, V_k$ be the parts that $u,v, w$ lies in respectively. Then $(V_i, V_j)$ and $(V_i, V_k)$ are $\delta^{1/3}$-dense and $(V_j,V_k)$ is $\delta^{1/3}$-sparse. By Claim 1, we have
$$\lambda_n(G) \leq -\frac{\min(|V_i|, |V_j|, |V_k|)}{10} \leq -\frac{\delta}{10t} n = - c_{\epsilon, \gamma} \cdot n.$$
Recalling that $t = \ceil{2\gamma^{-10} + 1}^{\gamma^{-2}}$, the leading constant $c_{\epsilon, \gamma} = \frac{\delta}{10t}$ depends only on $\epsilon$ and $\gamma$, as desired.

\textbf{Case 2: }Suppose $G$ contains no cherry with vertices $(u,v, w)$ such that $(u, v), (v, w), (w, u)$ are all good pairs. Then each cherry in $G$ must contain a bad pair. By Claim 3, $G$ contains at most
$$12\delta^{1/3} n^3 \leq \delta_{K_{1, 2}} n^3$$
cherries. By \cref{lem:cherries}, $G$ is $\epsilon$-close to a disjoint union of cliques, as desired.

\section{\cref{thm:discrepancy}: The key bounds}
\label{sec:discrepancy-1}
Let's move on to the surplus results. In this section, we carry out the key step in the proof of \cref{thm:discrepancy} by proving an analog of \cref{lem:recursive} with surplus replacing the least eigenvalue. 

First, we recap a couple of useful tools developed in recent years for studying surplus. In particular, we need two upper bounds on surplus in terms of the energy and the least eigenvalue. Using a ground--breaking result of Alon, Makarychev, Makarychev and Naor \cite{AMMN}, R\"{a}ty and Tomon \cite{RT2024} showed that the surplus, which is NP--hard to compute exactly, is well-approximated by the following semi-definite program computable in polynomial time
$$\surp^*(G) := \sup_{\substack{M \succ 0  \\ M_{ii} \leq 1, \forall i \in V}} \frac{1}{2} \langle -A_G, M \rangle$$
where $M \succ 0$ means that $M$ is positive semi--definite. For the sake of completeness, we repeat some of their arguments here.

Observe that $\surp(G)$ is the objective of this program when $M$ is restricted to the matrices $\{ \bv \bv^\top : \bv \in \{\pm 1\}^V\}$. Hence, we have
$$\frac{\surp^*(G)}{K(G)} \leq \surp(G) \leq \surp^*(G)$$
where $K(G)$ is the Grothendieck constant \cite[Definition 2.1]{AMMN} of $G$. Alon, Makarychev, Makarychev and Naor showed that for any loopless graph $G$, we have $K(G) = O(\log n)$ (in fact, this particular result was shown earlier by Charikar and Wirth \cite{CharikarWirth04}). As a corollary, $\surp(G)$ and $\surp^*(G)$ are within a multiplicative factor of $\log n$ of each other. 

R\"{a}ty and Tomon \cite[Lemma 3.2]{RT2024}  used the test matrix 
$$M = \sum_{i: \lambda_i \leq 0} \bv_i \bv_i^\top$$
to show a beautiful relation between surplus and energy.
\begin{lemma}[Energy bound]
\label{lem:surplus-energy}
For any graph $G$, we have
$$
\surp^*(G) \geq \frac{1}{2} \sum_{i: \lambda_i \leq 0} -\lambda_i  = \frac{1}{4} \cE(G).
$$
\end{lemma}
This energy bound can replace the least eigenvalue bound for most arguments in \cref{thm:main}. However, one particular step becomes tricky. The proof of \cref{lem:recursive} depends on \eqref{eq:RHS-2}, an upper bound on the quantity
\begin{equation}
\label{eq:strange-sum}
\sum_{i \in L_0, j\notin L_0} \lambda_i (-\lambda_j) \langle \bv_i \circ \bv_j, q\rangle^2.
\end{equation}
We are unable to achieve a tight analog of \eqref{eq:RHS-2} with an upper bound on $\cE(G)$ alone. To make some progress, we first show a bound on the least eigenvalue for graphs with small surplus. It roughly says that if a graph has surplus close to $n$, then its least eigenvalue is at most $n^{2/3}$. We are not sure if the exponent $\frac{2}{3}$ is tight, but it cannot be improved beyond $\frac{1}{2}$, as shown by removing a $K_{\sqrt{n}}$ from $K_n$. Optimizing this lemma might be an interesting starting point for improving \cref{thm:discrepancy}.
\begin{lemma}
    \label{lem:least-pointwise-surplus}
  For any graph $G$ on $n$ vertices, we have
  $$-\lambda_n \leq (2n \surp^*(G))^{1/3}.$$
\end{lemma}
\begin{proof}
  We consider the test matrix (an analogous test matrix appears in \cite{RST2023} as well)
    $$M = \frac{1}{n}\sum_{j \notin L_0} \lambda_j^2 \bv_j \bv_j^\top.$$
    For each vertex $k \in V$, we have
    $$M_{kk} = \frac{1}{n}\sum_{j \notin L_0} (\lambda_j \bv_{jk})^2 = \frac{1}{n}\sum_{j \notin L_0} \langle \bv_j, \bw_k \rangle^2$$
    where $\bw_k$ is the indicator vector of the neighbor of $k$. As $\{\bv_j\}_{j = 1}^n$ form an orthonormal basis, by Parseval's identity, we have
    $$M_{kk} \leq \frac{1}{n} \norm{\bw_k}_2^2 \leq 1.$$
    Therefore, we have
    $$\surp^*(G) \geq \frac{1}{2}\langle -A_G, M \rangle = \frac{1}{2n} \sum_{j \notin L_0} (-\lambda_j)^3.$$
    As each term on the right--hand side is non--negative, we discard all but one term to obtain
    $$(-\lambda_n)^3 \leq 2n \surp^*(G)$$
    as desired.
\end{proof}
This bound is non--trivial but substantially weaker than what we had in \cref{lem:recursive}. Using this bound, we prove the following weaker upper bound on \eqref{eq:strange-sum} in terms of the surplus.
\begin{lemma}
    \label{lem:key-bound-for-disc}
    For any graph $G$ on $n$ vertices, we have
    $$\sum_{i \in L_0, j\notin L_0} \lambda_i (-\lambda_j) \langle \bv_i \circ \bv_j, q\rangle^2 \leq n^{1/3} (2\surp^*(G))^{4/3} \norm{q}_{\infty}^2.$$
\end{lemma}
\begin{proof} 
As in the proof of \cref{lem:recursive}, we have
    $$\sum_{i \in L_0, j\notin L_0} \lambda_i (-\lambda_j) \langle \bv_i \circ \bv_j, q\rangle^2 \leq -\lambda_{n} \sum_{i \in L_0} \lambda_i \norm{\bv_i \circ q}_2^2 \leq n^{1/3}(2\surp^*(G))^{1/3} \sum_{i \in L_0} \lambda_i \norm{\bv_i \circ q}_2^2.$$
    Observing that
    $$\norm{\bv_i \circ q}_2^2  \leq \norm{q}_\infty^2 \norm{\bv_i}_2^2 = \norm{q}_\infty^2$$
    we have
    $$\sum_{i \in L_0} \lambda_i \norm{\bv_i \circ q}_2^2 \leq \norm{q}_\infty^2 \sum_{i \in L_0} \lambda_i = \frac{1}{2}\norm{q}_\infty^2 \cdot \cE(G).$$
    Applying \cref{lem:surplus-energy} gives
    $$\sum_{i \in L_0} \lambda_i \norm{\bv_i \circ q}_2^2 \leq 2\surp^*(G)\norm{q}_\infty^2$$
    so we have
    $$\sum_{i \in L_0, j\notin L_0} \lambda_i (-\lambda_j) \langle \bv_i \circ \bv_j, q\rangle^2 \leq n^{1/3} (2\surp^*(G))^{4/3} \norm{q}_{\infty}^2$$
    as desired.
\end{proof}
In order to apply \cref{lem:key-bound-for-disc}, we need an upper bound on $\norm{q}_\infty$. A tight upper bound is not available if we simply choose $q \sim N(0, I_W)$ as in the proof of \cref{lem:recursive}. To guarantee this bound, we clip $q$.
\begin{definition}
  For a vector $q \in \RR^n$ and some \emph{clipping factor} $\beta > 1$, define the \emph{clipping} of $q$, denoted $\mathsf{T}_{\beta} q$, as follows. Set $x = \frac{1}{\sqrt{n}} \norm{q}_2$. For each $i \in [n]$, define the $i$-th coordinate of $\mathsf{T}_{\beta} q$ as
  $$(\mathsf{T}_{\beta} q)_i = \begin{cases}
      \beta x, q_i > \beta x \\
      q_i, q_i \in [-\beta x, \beta x] \\
      -\beta x, q_i < -\beta x
  \end{cases}.$$
\end{definition}
Clearly, we have $\norm{\mathsf{T}_{\beta} q}_2 \leq \norm{q}_2$ and $\norm{q - \mathsf{T}_{\beta} q}_2 \leq \norm{q}_2$. The next lemma shows that clipping $q$ changes the quantities of interest in the proof of \cref{lem:recursive} by a small factor.
\begin{lemma}
\label{lem:effect-of-truncation}
Let $G$ be a graph on $n$ vertices, let $T \in [0, n]$, and let $W$ be the vector space spanned by $\{\bv_i \circ \bv_j: i, j \in L_T\}$. Let $q \sim N(0, I_W)$ be as in the proof of \cref{lem:recursive}. Take the clipping factor
$$\beta = \frac{2n^4}{T^4}.$$
1) For any $i, j \in L_T$, we have
$$\EE \langle \bv_i \circ \bv_j, \mathsf{T}_{\beta} q \rangle^2 \geq \left( \norm{\bv_i \circ \bv_j}_2 - \frac{1}{2\sqrt{n}}\right)_+^2$$
where $x_+^2$ is defined as $\max(x, 0)^2$.

2) For any $i \in L_{\frac{T^2}{8n}}$, we have 
$$\EE \langle \bv_i, \mathsf{T}_{\beta} q \rangle^2 \leq 25.$$
\end{lemma}
\begin{proof}
    Both parts of the lemma crucially depends on the observation that eigenvectors of large eigenvalues are flat. More precisely, we have

    \emph{Flatness Property: }For any $i$, we have $\norm{\bv_i}_\infty \leq \frac{\sqrt{n}}{\abs{\lambda_i}}$.

    \emph{Proof: }For each coordinate $k \in [n]$, we have
    $$\abs{\bv_{ik}} = \frac{1}{\abs{\lambda_i}} \abs{\sum_{\ell: k\ell \in E} \bv_{i\ell}}.$$
    By the Cauchy-Schwarz inequality, we have
    $$\abs{\sum_{\ell: k\ell \in E} \bv_{i\ell}} \leq \sqrt{n} \norm{\bv_i}_2 = \sqrt{n}$$
    from which we conclude the desired property. $\qed$
    
    We now prove part 1). By the triangle inequality, we have
    $$\left(\EE \langle \bv_i \circ \bv_j, \mathsf{T}_{\beta} q \rangle^2\right)^{1/2} \geq \left(\EE \langle \bv_i \circ \bv_j, q \rangle^2\right)^{1/2} - \left(\EE \langle \bv_i \circ \bv_j, q - \mathsf{T}_{\beta} q \rangle^2\right)^{1/2}.$$
    By \cref{prop:Gaussian-basic}, we have
    $$\EE \langle \bv_i \circ \bv_j, q \rangle^2 = \norm{\bv_i \circ \bv_j}_2^2.$$
    On the other hand, let $H$ be the set of indices $i$ with $q_i \neq (\mathsf{T}_{\beta} q)_i$. For each $i \in H$, we have
    $$\abs{q_i}^2 \geq \beta^2 x^2 = \frac{\beta^2}{n} \norm{q}_2^2.$$
    Therefore, almost surely we have
    $$\abs{H} \leq \frac{\norm{q}_2^2}{\frac{\beta^2}{n} \norm{q}_2^2} = \frac{n}{\beta^2}.$$
    Applying the flatness property to $\bv_i$ and $\bv_j$, almost surely we have
    $$\abs{\langle \bv_i \circ \bv_j, q - \mathsf{T}_{\beta} q \rangle} = \sum_{k \in H} \abs{\bv_{ik} \bv_{jk} (q - \mathsf{T}_{\beta} q )_k} \leq \frac{n}{T^2} \sum_{k \in H} \abs{(q - \mathsf{T}_{\beta} q )_k} \leq \frac{n}{T^2} |H|^{1/2} \norm{q - \mathsf{T}_{\beta} q}_2 \leq \frac{n^{3/2}}{T^2 \beta} \norm{q - \mathsf{T}_{\beta} q}_2.$$
    Thus we obtain
    $$\EE \langle \bv_i \circ \bv_j, q - \mathsf{T}_{\beta} q \rangle^2 \leq \frac{n^3}{T^4 \beta^2} \EE \norm{q - \mathsf{T}_{\beta} q}_2^2 \leq \frac{n^3}{T^4 \beta^2} \EE \norm{q}_2^2 = \frac{n^3}{T^4 \beta^2} \dim W.$$
    Finally, we have
    $$|L_T| \leq \frac{1}{T^2} \sum_{i = 1}^n \lambda_i^2 \leq \frac{n^2}{T^2}$$
    so
    $$\dim W \leq |L_T|^2 \leq \frac{n^4}{T^4}.$$
    Thus we conclude that
    $$\EE \langle \bv_i \circ \bv_j, q - \mathsf{T}_{\beta} q \rangle^2 \leq \frac{n^7}{T^8 \beta^2}.$$
    By our choice of $\beta = \frac{2n^4}{T^4}$, we obtain
    $$\EE \langle \bv_i \circ \bv_j, q - \mathsf{T}_{\beta} q \rangle^2 \leq \frac{1}{4n}$$
    so we have
    $$\left(\EE \langle \bv_i \circ \bv_j, \mathsf{T}_{\beta} q \rangle^2\right)^{1/2} \geq \left(\EE \langle \bv_i \circ \bv_j, q \rangle^2\right)^{1/2} - \left(\EE \langle \bv_i \circ \bv_j, q - \mathsf{T}_{\beta} q \rangle^2\right)^{1/2} \geq \norm{\bv_i \circ \bv_j}_2 - \frac{1}{2  \sqrt{n}}$$
    as desired. This completes the proof of part 1). 

    The proof of part 2) is similar. By the triangle inequality, we have
    $$\left(\EE \langle \bv_i, \mathsf{T}_{\beta} q \rangle^2\right)^{1/2} \leq \left(\EE \langle \bv_i, q \rangle^2\right)^{1/2} + \left(\EE \langle \bv_i, q - \mathsf{T}_{\beta} q \rangle^2\right)^{1/2}.$$
    By \cref{prop:Gaussian-basic}, we have
    $$\EE \langle \bv_i, q \rangle^2 \leq \norm{\bv_i}_2^2 = 1.$$
    On the other hand, as $i \in L_{\frac{T^2}{8n}}$, we have $\lambda_i \geq \frac{T^2}{8n}$, so the flatness property implies that
    $$\abs{\langle \bv_i, q - \mathsf{T}_{\beta} q \rangle} \leq \sum_{k \in H} \abs{\bv_{ik}} \abs{(q - \mathsf{T}_{\beta} q)_k} \leq \frac{8n\sqrt{n}}{T^2} \cdot |H|^{1/2} \cdot \norm{q - \mathsf{T}_{\beta} q}_2 \leq \frac{8n^{2}}{T^2 \beta} \norm{q - \mathsf{T}_{\beta} q}_2.$$
    Taking the expectation, we have
    $$\EE \abs{\langle \bv_i, q - \mathsf{T}_{\beta} q \rangle}^2 \leq \frac{64n^4}{T^4 \beta^2} \EE \norm{q - \mathsf{T}_{\beta} q}_2^2 \leq \frac{64n^4}{T^4 \beta^2} \EE \norm{q}_2^2 = \frac{64n^4}{T^4 \beta^2} \dim W.$$
    Substituting the estimate $\dim W \leq \frac{n^4}{T^4}$, we obtain
    $$\EE \abs{\langle \bv_i, q - \mathsf{T}_{\beta} q \rangle}^2 \leq \frac{64n^8}{T^8 \beta^2}.$$
    Using our choice of $\beta = \frac{2n^4}{T^4}$, we conclude that
    $$\EE \abs{\langle \bv_i, q - \mathsf{T}_{\beta} q \rangle}^2 \leq 16.$$
    So we have
    $$\left(\EE \langle \bv_i, \mathsf{T}_{\beta} q \rangle^2\right)^{1/2} \leq 1 + 4 = 5$$
    giving part 2) as desired.
\end{proof}
We are now ready to prove an analog of \cref{lem:recursive} for graphs with small surplus.
\begin{lemma}
    \label{lem:recursive-II}
    Set $c = \frac{1}{99}$. Assume $G$ is a graph on $n \gg 1$ vertices with $\surp^*(G) \leq \frac{1}{2} n^{1 + c}$. Then for every threshold $T \geq n^{1 - 2c}$, we have
    $$S_T^2 \leq 250 n S_{\frac{T^2}{8n}}.$$
\end{lemma}
\begin{proof}
Let $W$ be the subspace of $\RR^n$ spanned by the vectors $\{\bv_i \circ \bv_j: i, j \in S_T\}$. Let $q\sim N(0, I_W)$ be the standard Gaussian vector on $W$, and take the cutoff
$$\beta = \frac{2n^4}{T^4} \leq 2n^{8c}.$$
As in the proof of \cref{lem:recursive}, we start with the identity 
$$\sum_{i = 1}^n \lambda_i \langle \bv_i, \mathsf{T}_{\beta} q\rangle^2 = \sum_{i, j = 1}^n \lambda_i \lambda_j \langle \bv_i \circ \bv_j, \mathsf{T}_{\beta} q\rangle^2.$$
We estimate the expectation of both sides. For the left--hand side, by Parseval's identity we have
$$\sum_{i \notin L_{\frac{T^2}{8n}}}\lambda_i \langle \bv_i, \mathsf{T}_{\beta} q\rangle^2 \leq \frac{T^2}{8n} \sum_{i = 1}^n \langle \bv_i,  \mathsf{T}_{\beta} q\rangle^2 = \frac{T^2}{8n} \norm{\mathsf{T}_{\beta} q}_2^2 \leq \frac{T^2}{8n} \norm{q}_2^2.$$
Therefore, we have
$$\sum_{i = 1}^n \lambda_i \langle \bv_i, \mathsf{T}_{\beta} q\rangle^2 \leq \sum_{i \in L_{\frac{T^2}{8n}}} \lambda_i \langle \bv_i, \mathsf{T}_{\beta} q\rangle^2 + \frac{T^2}{8n} \norm{q}_2^2.$$
Taking expectation, we have
$$\EE \sum_{i = 1}^n \lambda_i \langle \bv_i, \mathsf{T}_{\beta} q\rangle^2 \leq \sum_{i \in L_{\frac{T^2}{8n}}} \lambda_i \EE \langle \bv_i, \mathsf{T}_{\beta} q\rangle^2 + \frac{T^2}{8n} \EE \norm{q}_2^2$$
In light of \cref{lem:effect-of-truncation} (2) and \cref{prop:Gaussian-basic}(1), we obtain
\begin{equation}
    \label{eq:LHS-II}
    \EE \sum_{i = 1}^n \lambda_i \langle \bv_i, q\rangle^2 \leq 25 S_{\frac{T^2}{8n}} + \frac{T^2}{8n} \dim W.
\end{equation}
We move on to the right--hand side. Dropping non-negative terms, we have
$$\sum_{i, j = 1}^n \lambda_i \lambda_j \langle \bv_i \circ \bv_j, \mathsf{T}_{\beta} q\rangle^2 \geq \sum_{i, j \in L_T} \lambda_i \lambda_j \langle \bv_i \circ \bv_j, \mathsf{T}_{\beta} q\rangle^2 - 2\sum_{i \in L_0, j\notin L_0} \lambda_i (-\lambda_j) \langle \bv_i \circ \bv_j, \mathsf{T}_{\beta} q\rangle^2.$$
By \cref{lem:effect-of-truncation} (1), we have
$$\EE \langle \bv_i \circ \bv_j, \mathsf{T}_{\beta} q\rangle^2 \geq \left( \norm{\bv_i \circ \bv_j}_2 - \frac{1}{2\sqrt{n}}\right)_+^2$$
hence
$$\EE \sum_{i, j \in L_T} \lambda_i \lambda_j \langle \bv_i \circ \bv_j, \mathsf{T}_{\beta}q\rangle^2 = \sum_{i, j \in L_T} \sum_{k = 1}^n \lambda_i \lambda_j \left( \norm{\bv_i \circ \bv_j}_2 - \frac{1}{2\sqrt{n}}\right)_+^2.$$
By the AM-GM inequality, we have
$$\left( \norm{\bv_i \circ \bv_j}_2 - \frac{1}{2\sqrt{n}}\right)_+^2 \geq \frac{1}{2} \norm{\bv_i \circ \bv_j}_2^2 - \frac{1}{4n}.$$
By the Cauchy-Schwarz inequality, we have
$$
\sum_{i, j \in L_T} \lambda_i \lambda_j \norm{\bv_i \circ \bv_j}_2^2 = \sum_{k = 1}^n \left(\sum_{i \in L_T} \lambda_i \bv_{ik}^2\right)^2 \geq \frac{1}{n} \left(\sum_{k = 1}^n \sum_{i \in L_T} \lambda_i \bv_{ik}^2\right)^2 = \frac{1}{n} S_T^2.
$$
Therefore, we conclude that
\begin{equation}
    \label{eq:RHS-1-II}
\EE \sum_{i, j \in L_T} \lambda_i \lambda_j \langle \bv_i \circ \bv_j, \mathsf{T}_{\beta}q\rangle^2 \geq \sum_{i, j \in L_T} \sum_{k = 1}^n \lambda_i \lambda_j \left( \frac{1}{2} \norm{\bv_i \circ \bv_j}_2^2 - \frac{1}{4n}\right) \geq \frac{1}{2n} S_T^2 - \frac{1}{4n} S_T^2 = \frac{1}{4n}S_T^2.
\end{equation}
Finally, by \cref{lem:key-bound-for-disc} we have
$$\sum_{i \in L_0, j\notin L_0} \lambda_i (-\lambda_j) \langle \bv_i \circ \bv_j, \mathsf{T}_{\beta} q\rangle^2 \leq n^{1/3} (2\surp^*(G))^{4/3} \norm{\mathsf{T}_{\beta} q}_{\infty}^2 \leq n^{5/3 + 4c/3} \norm{\mathsf{T}_{\beta} q}_{\infty}^2.$$
By the definition of the clipping operation, we have
$$\norm{\mathsf{T}_{\beta} q}_{\infty}^2 \leq \beta^2 \cdot \frac{1}{n} \norm{q}_2^2 \leq \beta^2 \cdot \frac{1}{n} \norm{q}_2^2.$$
By \cref{prop:Gaussian-basic} and $\beta \leq 2n^{8c}$, we obtain
\begin{equation}
\label{eq:RHS-2-II}
\EE \sum_{i \in L_0, j\notin L_0} \lambda_i (-\lambda_j) \langle \bv_i \circ \bv_j, q\rangle^2 \leq n^{2/3 + 4c/3} \beta^2 \EE \norm{q}_2^2 \leq 4n^{2/3 + 28c/3} \dim(W).
\end{equation}
Finally, combining \eqref{eq:LHS-II}, \eqref{eq:RHS-1-II} and \eqref{eq:RHS-2-II}, we obtain
$$25S_{\frac{T^2}{8n}} + \frac{T^2}{8n} \dim W \geq \frac{1}{4n}S_T^2 - 8n^{2/3 + 52c /3} \dim(W).$$
Recall that $\dim W \leq \abs{L_T}^2 \leq \frac{S_T^2}{T^2}$. Substituting into the above, we have
$$ 25nS_{\frac{T^2}{8n}} + \frac{1}{8} S_T^2 \geq \frac{1}{4} S_T^2 - \frac{8n^{2/3 + 52c/3}}{T^2} S_T^2.$$
Recalling the assumption that $T \geq n^{1 - 2c}$ and our choice $c = \frac{1}{99}$, we have
$$\frac{8n^{2/3 + 52c/3}}{T^2} \leq 8n^{64 c / 3 - 1/3} \leq 8n^{-0.01}  \leq 0.01.$$
Therefore, we conclude that
$$ 25nS_{\frac{T^2}{8n}} \geq \frac{1}{10} S_T^2$$
as desired.
\end{proof}
\section{\cref{thm:discrepancy}: Finishing the proof}
\label{sec:discrepancy-2}
Using the critical \cref{lem:recursive-II} in the preceding section, we can prove \cref{thm:discrepancy} by an argument analogous to how \cref{thm:main} follows from \cref{lem:recursive}.

As $\surp^*(G) = O(\surp(G) \cdot \log n)$, the following result implies \cref{thm:discrepancy}.
\begin{theorem}
    Let $c = \frac{1}{99}$. For any $\epsilon > 0$, the following holds for all $n \gg_{\epsilon} 1$. If $G$ is a graph on $n$ vertices that is $\epsilon$-far from a disjoint union of cliques, then we have
    $$\surp^*(G) \geq \frac{1}{2}n^{1 + c}.$$
\end{theorem}
\begin{proof}
    Suppose for the sake of contradiction that $\surp^*(G) < \frac{1}{2}n^{1 + c}$. Take $\delta = \delta(\epsilon)$ as in \cref{lem:spectrum-to-clique}.

    Trying to apply \cref{lem:solving-recursion-general}, we check that

    1) By \cref{lem:least-pointwise-surplus}, we have $-\lambda_n \leq (2n \surp^*(G))^{1/3} \leq n^{\frac{2 + c}{3}}$.

    2) By \cref{lem:surplus-energy}, we have
    $$\sum_{i} \abs{\lambda_i} = \cE(G) \leq 4 \surp^*(G) \leq 2n^{1 + c}.$$
    3) By \cref{lem:recursive-II}, for $T \geq n^{1 - 2c}$ we have
    $$S_T^2 \leq 250n S_{\frac{T^2}{8n}}.$$
    Hence, \cref{lem:solving-recursion-general} applies with
    $$(p,q,r,C) = \left(\frac{2 + c}{3}, 1 + c, 1-2c, 250\right)$$
    giving
    $$s = \frac{q - 1}{1 - r} = \frac{1}{2}.$$
    Set $H = 10^{-20}\delta^2 n$, and note that $H \geq n^{p + q - 1}$ given $n \gg_{\epsilon} 1$. Therefore, we obtain
    $$\sum_{i \notin L_H} \lambda_i^2 \leq \frac{2 \cdot (250)^{3/2}}{\frac{1}{2}} \cdot n^{3/2} H^{1/2} \leq \delta n^2.$$
    Hence, we can apply \cref{lem:spectrum-to-clique} with $\gamma = 10^{-20}\delta^2$. Thus, one of the following must hold.
    \begin{enumerate}
        \item $G$ is $\epsilon$-close to a disjoint union of cliques.
        \item The least eigenvalue of $G$ is at most $-c_{\epsilon, \gamma} \cdot n$, where $c_{\epsilon, \gamma} > 0$ is a constant depending only on $\epsilon$ and $\gamma$.
    \end{enumerate}
    However, we just showed that $-\lambda_{n} \leq n^{1 - 4c}$, so the latter does not hold for $n \gg_\epsilon 1$. Therefore, $G$ is $\epsilon$-close to a disjoint union of cliques, contradiction.
\end{proof}
\section{\cref{thm:polynomial}: A fast density increment}
\label{sec:ABKS?}
In this section, we prove \cref{thm:polynomial} and \cref{cor:ABKS?}. Our arguments involve passing to induced subgraphs. Therefore, we will repeatedly use the following basic yet important property of surplus. 
\begin{proposition}
    \label{prop:surplus--monotone}
    If $H$ is an induced subgraph of $G$, then we have
    $$\surp(G) \geq \surp(H), \surp^*(G) \geq \surp^*(H).$$
\end{proposition}
\begin{proof}
    For the former inequality, see e.g. \cite[Lemma 2.5]{RT2024}. For the latter inequality, recall that
    $$\surp^*(G) := \sup_{\substack{M \succ 0  \\ M_{ii} \leq 1, \forall i \in V}} \frac{1}{2} \langle -A_G, M \rangle.$$
    For any feasible solution $M$ to the SDP for $H$, let $M'$ be the matrix with rows and columns indexed by $V(G)$, and entries defined by
    $$M'_{ij} = \begin{cases}
        M_{ij}, i, j \in V(H) \\
        0, \text{otherwise}
    \end{cases}.$$
    Then $M'$ is a feasible solution to the SDP for $G$, and $\langle -A_H, M \rangle = \langle -A_G, M' \rangle$. So we conclude that $\surp^*(G) \geq \surp^*(H)$, as desired.
\end{proof}

Before we dive the proof of \cref{thm:polynomial}, we first show how \cref{cor:ABKS?} follows from \cref{thm:polynomial}.
\begin{proof}[Proof of \cref{cor:ABKS?} given \cref{thm:polynomial}]
    Suppose for the sake of contradiction that $\surp(G) <  m^{\frac{1}{2} + \frac{\rho}{10}}$. Set $d = \frac{1}{2} m^{\frac{1}{2} - \frac{\rho}{10}} - 1$. We divide into two cases.

    \textbf{Case 1:} Suppose $G$ is $d$-degenerate. Using the folklore fact that
    $$\surp(G) \geq \frac{m}{2\chi} \geq \frac{m}{2(d + 1)}$$
    we have
    $$\surp(G) \geq \frac{m}{2(d + 1)} = m^{\frac{1}{2} + \frac{\rho}{10}}$$
    contradiction.
    
    \textbf{Case 2:} Suppose $G$ is not $d$-degenerate. Then $G$ contains an induced subgraph $G'$ with minimum degree at least $d$. Let $n$ be the number of vertices in $G'$. Then we have
    $$d \leq n \leq \frac{2m}{d} \leq d^{1 + \frac{\rho}{4}} .$$
    Therefore, $G'$ has at least $\frac{1}{2} nd \geq n^{2 - \rho}$ edges, and we have
    $$\surp(G') \leq \surp(G) < m^{\frac{1}{2} + \frac{\rho}{10}} \leq d^{1 + \rho} \leq n^{1 + \rho}.$$
    Hence, for $m \gg_r 1$, \cref{thm:polynomial} shows that $G'$ must contain an induced subgraph $H$ on at least $n^{1 - 4\rho}$ vertices with edge density at least $1 - \frac{1}{r}$. By Tur\'{a}n's theorem $H$ must contain a copy of $K_r$, contradiction.
\end{proof}
The rest of this section is dedicated to the proof of \cref{thm:polynomial}. Recall that the edge density of a graph $G$ is defined as
$$p(G) = \frac{2m}{n^2}$$
where $m$ is the number of edges and $n$ is the number of vertices of $G$. The next lemma is the key density increment result that allows us to finish the proof. Density increment arguments are extensively used in combinatorics, for example in the proofs of Roth's theorem and Szemer\'{e}di's regularity lemma. In the study of surplus, density increment has led to new bounds for the log--rank conjecture \cite{SudakovTomon24}. The main challenge in our setting is to obtain polynomial--type bounds, requiring a fast density increment made possible by spectral information revealed by the recursive inequality. We will use this lemma to find a large induced subgraph of $G$ with density $\Omega(1)$, then apply \cref{thm:discrepancy} to boost the density to $1 - \epsilon$ as desired.
\begin{lemma}
    \label{lem:density-increment}
    Let $\gamma = \frac{1}{200}$. Let $G_0$ be a graph on $n_0 \gg 1$ vertices with edge density at least $n_0^{-\gamma}$, such that 
    $$\surp^*(G_0) \leq \frac{1}{2} n_0^{1 + \gamma}.$$
    Then one of the following holds.

    1) $G_0$ contains an induced subgraph $G_1$ on $n' \geq \frac{n_0}{8}$ vertices with $p(G_1) \geq 2 p(G_0)$.

    2) $G_0$ contains an induced subgraph $G_1$ on $n' \geq p(G_0) n_0$ vertices with $p(G_1) \geq 10^{-10} \sqrt{p(G_0)}$.
\end{lemma}

\begin{proof}
    Write $p_0 = p(G_0)$. Our first goal is to obtain control on the maximum degree. We repeat remove a vertex of $G_0$ with degree at least $4 p_0 n_0$, until no such vertex remains. Let $R$ be the set of removed vertices, and let $G$ be the subgraph of $G_0$ induced on the remaining vertices. As each step removes at least $4p_0n_0$ edges, we have $|R| \leq \frac{\frac{1}{2}p_0 n_0^2}{4p_0n_0} = \frac{n_0}{8}$.

    First, assume $p(G) \leq \frac{p_0}{4}$. Then the number of edges in $G = G_0[\bar{R}]$ is
    $$e(G_0[\bar{R}]) \leq \frac{1}{2} p(G)n_0^2 \leq \frac{p_0n_0^2}{8}.$$
    Furthermore, by the bound on surplus, the number of edges between $R$ and $\bar{R}$ is at most
    $$e(G_0[R, \bar{R}]) \leq \frac{e(G_0)}{2} + \surp(G_0) \leq \frac{p_0n_0^2}{4} + \frac{1}{2} n_0^{1 + \gamma} \leq \frac{p_0 n_0^2}{3}.$$
    Therefore, the total number of edges in $G_0[R]$ is at least
    $$e(G_0[R]) \geq \frac{p_0n_0^2}{2} - \frac{p_0 n_0^2}{3} - \frac{p_0n_0^2}{8} = \frac{1}{24} p_0 n_0^2.$$
    We let $W$ be any set of $\ceil{\frac{n_0}{8}}$ vertices containing $R$, and let $G_1 = G_0[W]$. Then we have
    $$p(G_1) \geq \frac{2e(G_0[R])}{|W|^2} \geq \frac{\frac{1}{24} p_0 n_0^2}{\frac{1}{63}n_0^2} \geq 2p_0$$
    so $G_1$ satisfies 1), as desired.

    In the rest of the argument, we assume $p(G) \geq \frac{p_0}{4}$. Write $p = p(G)$ and let $n$ be the number of vertices in $G$. We will repeatedly use the facts that 
    $$n = n_0 - |R| \geq \frac{7n_0}{8} \text{\quad and \quad} p \geq \frac{p_0}{4} \geq \frac{1}{4}n_0^{-\gamma} \geq \frac{1}{8} n^{-\gamma}.$$ 
    Let $\Delta$ be the maximum degree of $G$. By assumption, we have
    $$\Delta \leq 4p_0 n_0 \leq 20 pn.$$
    From now on, all graph parameters such as $\lambda_i, S_T, \cdots$ refer to that of $G$, unless otherwise specified. By \cref{prop:surplus--monotone} we have
    $$\surp^*(G) \leq \surp^*(G_0) \leq \frac{1}{2}n_0^{1 + \gamma} \leq n^{1 + \gamma}.$$
    We again run the recursive inequality \cref{lem:recursive-II}. This time, we pay extra attention to the exponent $s$. Set $c = \frac{1}{99}$. We check that 

    1) By \cref{lem:least-pointwise-surplus}, we have $-\lambda_n \leq (2n \surp^*(G))^{1/3} \leq 2 n^{\frac{2 + \gamma}{3}}$.

    2) By \cref{lem:surplus-energy}, we have
    $$\sum_{i} \abs{\lambda_i} = \cE(G) \leq 4 \surp^*(G) \leq 4n^{1 + \gamma}.$$
    3) By \cref{lem:recursive-II}, for $T \geq n^{1 - 2c}$ we have
    $$S_T^2 \leq 250n S_{\frac{T^2}{8n}}.$$
    To avoid duplicate notation, we write $p_e$ for the parameter $p$ in \cref{lem:solving-recursion-general}. Hence, \cref{lem:solving-recursion-general} applies with
    $$(p_e,q,r,C) = \left(\frac{2 + \gamma}{3}, 1 + \gamma, 1-2c, 250\right)$$
    giving
    $$s = \frac{q - 1}{1 - r} = \frac{\gamma}{2c} \leq \frac{1}{4}.$$
    Let $H = \frac{p^{3/2} n}{10^6}$. Recalling that $p \geq \frac{1}{8} n^{-\gamma}$, for $n \gg 1$ we have
    $$n^{p_e + q - 1} = n^{2/3 + 4\gamma / 3} \leq H$$
    By \cref{lem:solving-recursion-general}, we obtain
    $$\sum_{i \notin L_H} \lambda_i^2 \leq \frac{2 \cdot (250)^{1.25}}{0.75} \cdot n^{\frac{5}{4}} H^{\frac{3}{4}} = \frac{2 \cdot (250)^{1.25}}{0.75} \cdot n^{\frac{5}{4}} \left(\frac{p^{3/2} n}{10^6}\right)^{\frac{3}{4}} \leq \frac{1}{2} pn^2.$$
    As $G$ has exactly $\frac{1}{2}pn^2$ edges, we conclude that
    $$\sum_{i \in L_H} \lambda_i^2 = pn^2 - \sum_{i \notin L_H} \lambda_i^2 \geq \frac{1}{2} pn^2.$$
    Previously, we invoked \cref{lem:spectrum-to-clique} at this point. As we mentioned in the introduction, the spectral partitioning argument is unlikely to give polynomial-type bound. Here we diverge from previous arguments by considering $t(G)$, the number of triangles in $G$. A familiar fact in spectral graph theory is that
    $$6t(G) = \tr(A_G^3) = \sum_{i} \lambda_i^3.$$
    It is clear that
    $$\sum_{i = 1}^n \lambda_i^3 \geq \sum_{i \in L_H} \lambda_i^3 - \sum_{i \notin L_0} (-\lambda_i)^3.$$
    On one hand, we have
    $$\sum_{i \in L_H} \lambda_i^3 \geq H \sum_{i \in L_H} \lambda_i^2 \geq \frac{p^{3/2} n}{10^6} \cdot \frac{1}{2} pn^2 = \frac{1}{2 \cdot 10^6} p^{5/2} n^3.$$
    On the other hand, recalling our estimates $\cE(G) \leq 2n^{1 + \gamma}, -\lambda_{n} \leq n^{\frac{2 + \gamma}{3}}$, and $p \geq n^{-2\gamma}$, we have
    $$\sum_{i \notin L_0} (-\lambda_i)^3 \leq \frac{1}{2} \cE(G) \lambda_{n}^2  \leq n^{7/3 + 10\gamma / 3} \leq \frac{1}{4 \cdot 10^6} p^{5/2} n^3.$$
    So we conclude that
    $$6t(G) \geq \frac{1}{4 \cdot 10^6} p^{5/2} n^3$$
    On the graph side, we have
    $$6t(G) = \sum_{v \in V} 2 e(G[N(v)])$$
    where $G[N(v)]$ is the subgraph of $G$ induced on the neighbor $N(v)$ of $v$. Thus, by the pigeonhole principle, there exists some $v_0 \in V$ such that
    $$2e(G[N(v_0)]) \geq \frac{1}{4 \cdot 10^6} p^{5/2} n^2.$$
    Recal the maximum degree bound $\Delta \leq 20 pn$. Also recall that $p \geq \frac{p_0}{4}$ and $n \geq \frac{7}{8}n_0$. Let $W$ be any set of $\floor{20pn}$ vertices containing $N(v_0)$, and let $G_1 = G[W]$. Then the number of vertices in $G_1$ is at least
    $$\floor{20pn} \geq \floor{20 \cdot \frac{7}{8} \cdot \frac{1}{4}p_0n_0} \geq p_0 n_0.$$
    Furthermore, we have
    $$p(G_1) \geq \frac{2e(G[N(v)])}{(20pn)^2} \geq \frac{1}{4 \cdot 10^6 \cdot 20^2} p^{1/2} \geq 10^{-10} p_0^{1/2}.$$
    So $G_1$ satisfies 2), as desired.
\end{proof}
\begin{proof}[Proof of \cref{thm:polynomial}]
    We construct a sequence of induced subgraphs of $G$. Let $G_1 = G$, and $G_{i + 1}$ be the graph $G_1$ obtained by applying \cref{lem:density-increment} with $G_0 = G_i$. We repeat this until either of the following happens.

    1) $G_i$ has edge density at least $10^{-100}$.

    2) $G_i$ no longer satisfies an assumption of \cref{lem:density-increment}.

    Let $p_i = p(G_i)$, and let $n_i$ denote the number of vertices of $G_i$. Write $p = p_1$. We can check that $p_{i + 1} \geq 2p_i$ always holds, so the process terminates. Let $G_1, G_2, \cdots, G_k$ be the sequence of graphs obtained. We now claim that 1) is how the process terminates. The key observation is that $p_i^3 n_i$ is non--decreasing in $i$. Indeed, consider each $i \leq k - 1$. If $G_i$ falls into case 1) of \cref{lem:density-increment}, then we have
    $$p_{i + 1} \geq 2p_i, n_{i + 1} \geq \frac{n_i}{8}$$
    so $p_{i + 1}^3 n_{i + 1} \geq p_i^3 n_i$. If $G_i$ falls into case 2) of \cref{lem:density-increment}, then we have
    $$p_{i + 1} \geq 10^{-10} \sqrt{p_i}, n_{i + 1} \geq p_i n_i.$$
    So we have
    $$p_{i + 1}^3 n_{i + 1} \geq 10^{-30} p_i^{3/2 + 1} n_i \geq p_i^3 n_i$$
    where the last inequality uses $p_i \leq 10^{-100}$. Therefore, $p_i^3 n_i$ is non--decreasing in $i$, so we have
    $$p_k^3 n_k \geq p^3 n.$$
    Recall that $G$ satisfies $p \geq n^{-\rho}$ and $\surp(G) \leq n^{1 + \rho}$. Hence, we obtain
    $$n_k \geq p^3 n \geq n^{1 - 3\rho}$$
    and
    $$p_k \geq p = n^{-\rho}.$$
    By our choice of $\rho = \frac{1}{1000}$ and $\gamma = \frac{1}{200}$, we observe that $n_k \geq \sqrt{n}$ is sufficiently large for \cref{lem:density-increment} to apply. We also have
    $$p_k \geq n_0^{-\rho} \geq n_k^{-\frac{\rho}{1 - 3\rho}} \geq n_k^{-\gamma}$$
    and by \cref{prop:surplus--monotone}
    $$\surp^*(G_k) \leq \surp^*(G_0) \leq O(\log n \cdot \surp(G_0)) \leq O(\log n \cdot n^{1 + \rho}) \leq \frac{1}{2} n_k^{1 + \gamma}.$$
    Therefore, \cref{lem:density-increment} applies to $G_k$. As the process terminates, we must have 
    $$p_k \geq 10^{-100}.$$
    We just showed that
    $$\surp(G_k) \leq \surp^*(G_k) \leq \frac{1}{2} n_k^{1 + \gamma}.$$
    Hence, \cref{thm:discrepancy} shows that $G_k$ is $10^{-300} \epsilon^3$-close to a disjoint union of clique, given that $n \gg_\epsilon 1$.
    
    Let $\cC_0, \cC_1, \cdots, \cC_\ell$ be the disjoint cliques, of size $c_0 \geq \cdots \geq c_{\ell}$. As $G_k$ has exactly $\frac{1}{2} p_k n_k^2$ edges, we have
    $$\binom{c_0}{2} + \cdots + \binom{c_\ell}{2} \geq \frac{1}{2} p_k n_k^2 - 10^{-300} \epsilon^3 n_k^2.$$
    As the cliques are disjoint, we have $c_0 + \cdots + c_{\ell} \leq n_k$. Therefore, we have
    $$c_0 \geq \frac{1}{2} p_k n_k - 10^{-300} \epsilon^3 n_k \geq \frac{1}{3} p_k n_k.$$
    We set $H = G_k[V(\cC_0)]$. Then $H$ is missing at most $10^{-300} \epsilon^3 n_k^2$ edges, so we conclude that the induced subgraph $G_k[V(\cC_0)]$ of $G$ has edge density at least
    $$\frac{2\binom{c_0}{2} - 2 \cdot 10^{-300} \epsilon^3 n_k^2}{c_0^2} \geq 1 - \frac{c_0 + 2 \cdot 10^{-300} \epsilon^3 n_k^2}{c_0^2} \geq 1 - \frac{27 \cdot 10^{-300} \epsilon^3}{p_k^2} \geq 1 - \frac{27 \cdot 10^{-300} \epsilon^3}{10^{-200}} \geq 1 - \epsilon.$$
    Furthermore, the number of vertices of $H$ is at least
    $$\frac{1}{3}p_k n_k \geq \frac{1}{3} \cdot 10^{-100} \cdot n^{1 - 3\rho} \geq n^{1 - 4\rho}$$
    as desired.
\end{proof}
\bibliographystyle{plain}
\bibliography{bib}
\end{document}